\documentclass[final]{siamltex}
\usepackage{amssymb,amsmath}
\usepackage{bbding}
\usepackage{setspace}

\usepackage{chngcntr}
\counterwithout{equation}{section}

\usepackage[colorlinks]{hyperref}

\makeatletter
\renewcommand\@biblabel[1]{#1.}
\makeatother

\newcommand{\norm}[1]{\left\Vert#1\right\Vert}
\newcommand{\seq}[1]{\left<#1\right>}
\newcommand{\abs}[1]{\left\vert#1\right\vert}
\newcommand{\R}{{\rm I}\!{\rm R}} 
\newcommand{\N}{{\rm I}\!{\rm N}} 
\newcommand{\napprox}{\approx\hskip-.37cm /}

\newcommand{\et}{\mathcal{E}_{\{R_{\alpha}\}}^{\text{\rm{tot}}}}

 \newtheorem{thm}{Theorem}[section]
 \newtheorem{cor}[thm]{Corollary}
 \newtheorem{lem}[thm]{Lemma}
 
 \newtheorem{defn}[thm]{Definition}
 \newtheorem{rem}[thm]{Remark}

\title{Global Saturation of Regularization Methods for Inverse Ill-Posed Problems \thanks{This work was
supported in part by Consejo Nacional de Investigaciones
Cient\'{\i}ficas y T\'{e}cnicas, CONICET, through PIP 2010-2012 Nro. 0219,
by Universidad Nacional del Litoral, U.N.L., through project CAI+D
2009-PI-62-315, and by the Air Force Office of Scientific
Research, AFOSR, through Grant FA9550-10-1-0018.}}

\author{Terry Herdman\thanks{Interdisciplinary Center for Applied Mathematics, ICAM, Virginia
Tech, Blacksburg, VA 24061, USA ({\tt terry.herdman@vt.edu}).}
\and Ruben D. Spies$^\text{\Envelope,\,}$\thanks{Instituto de
Matem\'{a}tica Aplicada del Litoral, IMAL, CONICET-UNL, G\"{u}emes 3450,
S3000GLN, Santa Fe, Argentina and Departamento de Matem\'{a}tica,
Facultad de Ingenier\'{\i}a Qu\'{\i}mica, Universidad Nacional del Litoral,
Santa Fe, Argentina (\Envelope\,: {\tt
rspies@santafe-conicet.gov.ar}).} \and Karina G. Temperini
\thanks{Instituto de Matem\'{a}tica Aplicada del Litoral, IMAL,
CONICET-UNL, G\"{u}emes 3450, S3000GLN, Santa Fe, Argentina, and
Departamento de Matem\'{a}tica, Facultad de Humanidades y Ciencias,
Universidad Nacional del Litoral, Santa Fe, Argentina ({\tt
ktemperini@santafe-conicet.gov.ar}).}}
\begin{document}
\maketitle
%
%

\begin{abstract}
In this article the concept of saturation of an arbitrary
regularization method is formalized based upon the original idea
of saturation for spectral regularization methods introduced by
Neubauer \cite{Neubauer94}. Necessary and sufficient conditions
for a regularization method to have global saturation are
provided. It is shown that for a method to have global saturation
the total error must be optimal in two senses, namely as optimal
order of convergence over a certain set which at the same time,
must be optimal (in a very precise sense) with respect to the
error. Finally, two converse results are proved and the theory is
applied to find sufficient conditions which ensure the existence
of global saturation for spectral methods with classical
qualification of finite positive order and for methods with
maximal qualification. Finally, several examples of regularization
methods possessing global saturation are shown.
\end{abstract}



\pagestyle{myheadings} \thispagestyle{plain} \markboth{T. HERDMAN,
R. D. SPIES and K. G. TEMPERINI}{GLOBAL SATURATION OF
REGULARIZATION METHODS}

\section{Introduction}
\label{intro} Let $X, Y$ be infinite dimensional Hilbert spaces
and $T:X\rightarrow Y$ a bounded linear operator such that
$\mathcal{R}(T)$ is not closed. It is well known that under these
conditions, the linear operator equation
\begin{equation}\label{eq:0}
    Tx=y
\end{equation}
is ill-posed, in the sense that $T^\dag$, the Moore-Penrose
generalized inverse of $T$, is not bounded \cite{bookEHN}. The
Moore-Penrose generalized inverse is strongly related to the least
squares solutions of (\ref{eq:0}). In fact this equation has a
least squares solution if and only if $y \in
\mathcal{D}(T^\dag)\doteq \mathcal{R}(T)\oplus
\mathcal{R}(T)^\perp$. In that case, $x^\dag\doteq T^\dag y$ is
the least squares solution of minimum norm and the set of all
least-squares solutions of (\ref{eq:0}) is given by $x^\dag +
\mathcal{N}(T)$. If the problem is ill-posed then $x^\dag$ does
not depend continuously on the data $y$. Therefore, if instead of
the exact data $y$, a noisy observation $y^\delta$ is available,
with $\norm{y-y^\delta}\leq \delta$, where $\delta>0$ is small,
then it is possible that $T^\dag y^\delta$ does not even exist and
if it does, it will not necessarily be a good approximation of
$x^\dag$. This instability becomes evident when trying to
approximate $x^\dag$ by traditional numerical methods and
procedures. Thus, for instance, it is possible that the
application of the standard least squares approximating procedure
on an increasing sequence of finite-dimensional subspaces
$\{X_n\}$ of $X$ whose union is dense in  $X$, result in a
sequence $\{x_n\}$ of least squares solutions that does not
converge to $x^\dag$ (see \cite{ref:Seidman-80}) or, even worst,
that they diverge from $x^\dag$ with speed arbitrarily large (see
\cite{Spies-Temperini-2006}).

Ill-posed problems must be first regularized if one wants to
successfully attack the task of numerically approximating their
solutions. Regularizing an ill-posed problem such as (\ref{eq:0})
essentially means approximating the operator $T^\dag$ by a
parametric family of bounded operators $\{R_\alpha\}$, where
$\alpha$ is a regularization parameter.  If $y \in
\mathcal{D}(T^\dag)$, then the best approximate solution $x^\dag$
of (\ref{eq:0}) can be written as $x^\dag=\int_0^{\norm{T}^2+}
\frac{1}{\lambda} \, dE_\lambda T^\ast y$ where $\{E_\lambda\}$ is
the spectral family associated to the operator $T^\ast T$ (see
\cite{bookEHN}). This is mainly why many regularization methods
are based on spectral theory and consist in defining
$R_\alpha\doteq \int_0^{\norm{T}^2+} g_\alpha(\lambda) \,
dE_\lambda T^\ast$ where $\{g_\alpha\}$ is a family of functions
appropriately chosen such that for every $\lambda \in
(0,\norm{T}^2]$ there holds $\underset{\alpha \rightarrow
0^+}{\lim}g_\alpha(\lambda)=\frac{1}{\lambda}$.

However, it is important to emphasize that no mathematical trick
can make stable a problem that is intrinsically unstable. In any
case there is always loss of information. All a regularization
method can do is to recover the largest possible amount of
information about the solution of the problem, maintaining
stability. It is often said that the art of applying
regularization methods consist always in maintaining an adequate
balance between accuracy and stability. In 1994, however, Neubauer
(\cite{Neubauer94}) showed that certain spectral regularization
methods ``{\it saturate}'', that is, they become unable to
continue extracting additional information about the exact
solution even upon increasing regularity assumptions on it. In his
article, Neubauer introduced for the first time the idea of the
concept of ``{\it saturation}'' of regularization methods. This
idea referred to the best order of convergence that a method can
achieve independently of the smoothness assumptions on the exact
solution and on the selection of the parameter choice rule. Later
on, in 1997, Neubauer (\cite{ref:Neubauer-97}) showed that this
saturation phenomenon occurs in particular in the classical
Tikhonov-Phillips method. Saturation is however a rather subtle
and complex issue in the study of regularization methods for
inverse ill-posed problems and the concept has always escaped
rigorous formalization in a general context.

In 2001, Math\'{e} and Pereverzev (\cite{ref:Mathe-Pereverzev-2001})
used Hilbert scales to study the efficiency of approximating
solutions based on observations with noise (stochastic or
deterministic). In this context it is possible to quantify the
degree of ill-posedness and to obtain general conditions on
projection methods so that they attain optimal order of
convergence. These concepts were later extended by the same
authors (\cite{ref:Mathe-Pereverzev-2003}) who studied the optimal
convergence problem in variable Hilbert scales. In their article
they showed that there is a close relationship between the optimal
convergence of a method and the \textit{``a-priori''} regularity
(in terms of source sets) for spectral methods possessing
qualification of finite order. In 2009 Herdman et al.
(\cite{ref:Herdman-Spies-Temperini-2009}) introduced an extension
of the concept of qualification and introduced three different
levels: weak, strong and optimal. It was shown that weak
qualification extends the definition introduced by Math\'{e} and
Pereverzev (\cite{ref:Mathe-Pereverzev-2003}), in the sense that
the functions associated to orders of convergence and source sets
need not be the same.

In 2004, Math\'{e} (\cite{Mathe2004}) proposed general definitions
of the concepts of qualification and saturation for spectral
regularization methods. However, the concept of saturation defined
by Math\'{e} is not applicable to general regularization methods and
it is not fully compatible with the original idea of saturation
proposed by Neubauer in \cite{Neubauer94}. In particular, for
instance, the definition of saturation given in \cite{Mathe2004}
does not imply uniqueness and therefore, neither a best global order
of convergence.

In this article a general theory of global saturation for arbitrary
regularization methods is developed. It is shown that saturation
involves two aspects: on one hand (just like in Neubauer's original
idea) the characterization of the best global order of convergence
of the method, and on the other hand, the description of the source
set on which such a best global order of convergence is achieved.
Also, necessary and sufficient conditions are found for a
regularization method to have global saturation. In particular, it
is shown that for a method to have saturation, it is necessary that
the total error be optimal in two senses, namely as optimal order of
convergence over a certain set which at the same time, must satisfy
a certain optimality condition with respect to the error. Moreover,
an explicit form for the global saturation is given in terms of the
family of regularization operators and the operator associated to
the problem. Lastly, sufficient conditions are provided for spectral
methods with qualification of positive finite order and for spectral
methods with maximal qualification to have global saturation.

The organization of the paper is as follows. In Section 2
convergence bounds for regularization methods are defined and an
appropriate framework for their comparison is developed. In
Section 3 the concept of global saturation is introduced, its
relation with the total error and with convergence bounds is shown
and necessary and sufficient conditions for the existence of
global saturation are provided. In Section 4, a few converse
results are proved which, together with the results of Section 3,
are used to derive sufficient conditions for the existence of
global saturation for certain spectral regularization methods.

\section{Upper Bounds of Convergence for Regularization Methods}
\label{sec:1} In this section we define what we call {\it upper
bounds of convergence} for regularization methods and we develop
ways of comparing them on the same as well as on different sets.
Although this section may seem a little lengthy and tedious at a
first glance, it provides a solid mathematical background on which
all subsequent formalization and definitions are based upon.

In sequel and for convenience of notation, unless otherwise
specified, we shall assume that all subsets of the Hilbert space
$X$ under consideration are not empty and they do not contain
$x=0$. Also, without loss of generality we will assume that the
operator $T$ is invertible (since in the context of inverse
problems one always works with the Moore-Penrose generalized
inverse of $T$, the lack of injectivity is not really a problem).
Given $M\subset X$, we will denote with $\mathcal{F}_{M}$ the
collection of the following functions: we will say that $\psi \in
\mathcal{F}_{M}$ if there exists $a=a(\psi)>0$ such that $\psi$ is
defined in $M \times (0,a)$, with values in $(0,\infty)$ and it
satisfies the following conditions:
\begin{enumerate}
    \item $\underset{\delta\rightarrow 0^+}{\lim}\psi(x,\delta)=0$
    for all
$x \in
    M$, and
    \item  $\psi$ is continuous and increasing as a function of $\delta$
    in $(0,a)$ for each fixed $x\in M$.
\end{enumerate}
Roughly speaking, the collection $\mathcal{F}_{M}$ contains all
possible $\delta$-``{\it orders of convergence}" on the set $M$.

\begin{defn}\label{def:relac-fm}
Let $M\subset X$ and $\psi,\tilde{\psi} \in \mathcal{F}_{M}$.

i) We say that ``$\psi$ precedes $\tilde{\psi}$ on $M$'', and we
denote it $\psi \overset{M}{\preceq} \tilde{\psi}$, if there exist
a constant $r>0$ and $p:M\rightarrow (0,\infty)$ such that
$\psi(x,\delta) \leq p(x) \tilde{\psi}(x,\delta)$ for all $x \in
M$ and for every $\delta \in (0,r)$.

ii) We say that ``$\psi,\tilde{\psi}$ are \textit{equivalent on
$M$}'', and we denote it $\psi \overset{M}{\approx} \tilde{\psi}$,
if $\psi \overset{M}{\preceq} \tilde{\psi}$ and
$\tilde{\psi}\overset{M}{\preceq} \psi$.

iii) We say that ``$\psi$ strictly precedes  $\tilde{\psi}$ on
$M$'' and we denote it  $\psi\overset{M}{\prec}\tilde{\psi}$ if
$\psi \overset{M}{\preceq} \tilde{\psi}$ and $\underset{\delta
\rightarrow0^+}{\limsup}\,\frac{\psi(x,\delta)}{\tilde{\psi}(x,\delta)}=0$
for every $x \in M.$ \end{defn}

The following observations follow immediately from these
definitions.

$\bullet$ Given that $\psi,\tilde{\psi}>0$, in \textit{iii)} the
condition $\underset{\delta
\rightarrow0^+}{\limsup}\,\frac{\psi(x,\delta)}{\tilde{\psi}(x,\delta)}=0$
is equivalent to $\underset{\delta
\rightarrow0^+}{\lim}\,\frac{\psi(x,\delta)}{\tilde{\psi}(x,\delta)}=0$,
i.e., $\psi(x,\delta)=o(\tilde{\psi}(x,\delta))$ for $\delta
\rightarrow0^+$.

$\bullet$ The relation ``$\overset{M}{\preceq}$'' introduces a
partial ordering in $\mathcal{F}_M$ and ``$\overset{M}{\approx}$''
is an equivalence relation in $\mathcal{F}_{M}$.

$\bullet$ If
$\psi\overset{M}{\preceq}(\overset{M}{\approx},\overset{M}{\prec})\,\tilde{\psi}$
then
$\psi\overset{\tilde{M}}{\preceq}(\overset{\tilde{M}}{\approx},\overset{\tilde{M}}{\prec})\,\tilde{\psi}$
for every $\tilde{M}\subset M$.
With $\npreceq$, $\nprec$ and $\napprox$ we will denote the
negation of the relations $\preceq$, $\prec$ and $\approx$,
respectively.

\begin{lem}\label{lem:menor estricto}
Let $M\subset X$ and $\psi,\tilde{\psi}\in \mathcal{F}_M$. If
$\psi\overset{M}{\prec}\tilde{\psi}$ then
$\tilde{\psi}\overset{\tilde{M}}{\npreceq}\psi$ for every
$\tilde{M}\subset M$.
\end{lem}

\begin{proof}
For the contrareciprocal. Suppose there exists $\tilde{M} \subset
M$ such that $\tilde{\psi}\overset{\tilde{M}}{\preceq}\psi$. Let
$x_0 \in \tilde{M}$, then
$\tilde{\psi}\overset{\{x_0\}}{\preceq}\psi$, that is, there exist
constants $0<p<\infty$ and $r>0$ such that $\underset{\delta \,\in
(0,r)}{\sup}\frac{\tilde{\psi}(x_0,\delta)}{\psi(x_0,\delta)}\leq
p<\infty$. Then,
\begin{eqnarray*}
\underset{\delta\rightarrow0^+}{\limsup}\frac{\psi(x_0,\delta)}{\tilde{\psi}(x_0,\delta)}&\geq
&\underset{\delta\rightarrow0^+}{\liminf}\frac{\psi(x_0,\delta)}{\tilde{\psi}(x_0,\delta)}\geq
\underset{\delta\in
(0,r)}{\inf}\frac{\psi(x_0,\delta)}{\tilde{\psi}(x_0,\delta)} \\
   &=&\left(\underset{\delta\in
(0,r)}{\sup}\frac{\tilde{\psi}(x_0,\delta)}{\psi(x_0,\delta)}\right)^{-1}\geq\frac{1}{p}>0.
\end{eqnarray*}
Therefore, $\psi\overset{\{x_0\}}{\nprec}\tilde{\psi}$, from which
we deduce that $\psi\overset{M}{\nprec}\tilde{\psi}$, since
$x_0\in\tilde{M}\subset M$. \hfill
\end{proof}
\smallskip
\begin{defn}
Let $\{R_{\alpha}\}_{\alpha\in (0,\alpha_0)}$ be a family of
regularization operators for the problem $Tx=y$. We define the
``total error of $\{R_{\alpha}\}_{\alpha\in (0,\alpha_0)}$ at $x
\in X$ for a noise level $\delta$'' as
$$\et(x,\delta)\doteq
\underset{\alpha \in (0,\alpha_0)}{\inf}\;\underset{y^\delta \in
\overline{B_\delta(Tx)}}{\sup} \norm{R_\alpha y^\delta-x},$$ where
$\overline{B_\delta(Tx)}\doteq\{y \in Y: \norm{Tx-y}\leq
\delta\}$.
\end{defn}

Note that $\et$ is the error in the sense of the largest possible
discrepancy that can be obtained for an observation within the
noise level $\delta$, with any choice of the regularization
parameter $\alpha$.

\begin{rem}\label{rem:1}
Let $a>0$, $M\subset X$ and $\et:M\times (0,a)\rightarrow
(0,\infty)$ be the total error of $\{R_\alpha\}$. Then $\et \in
\mathcal{F}_{M}$. In fact, for each $x \in M$, $\et(x,\delta)$ is
increasing as a function of $\delta$, and given that
$\{R_\alpha\}$ is a family of regularization operators, it follows
that $\et(x,\delta)$ is continuous as a function of $\delta$ for
each fixed $x \in M$ and $\underset{\delta\rightarrow
0^+}{\lim}\et(x,\delta)=0$ for every $x \in M.$
\end{rem}

\begin{defn}\label{def-sat}
Let $\{R_{\alpha}\}_{\alpha \in (0,\alpha_0)}$ be a family of
regularization operators for the problem $Tx=y$, $M\subset X$ and
$\psi \in \mathcal{F}_M$.

i) We say that $\psi$ is an ``upper bound of convergence for the
total error of $\{R_\alpha\}_{\alpha \in (0,\alpha_0)}$ on $M$''
if $\et\overset{M}{\preceq}\psi$.

ii) We say that $\psi$ is a ``strict upper bound of convergence
for the total error of $\{R_\alpha\}_{\alpha \in (0,\alpha_0)}$ on
$M$'' if $\et\overset{M}{\prec}\psi$.

iii) We say that $\psi$ is an ``optimal upper bound of convergence
for the total error of $\{R_\alpha\}_{\alpha \in (0,\alpha_0)}$ on
$M$'' if $\et\overset{M}{\preceq}\psi$ and
\begin{equation*}
\underset{\delta
\rightarrow0^+}{\limsup}\,\frac{\et(x,\delta)}{\psi(x,\delta)}>0
\quad  \textrm{for every} \quad x \in M,
\end{equation*}
or equivalently, if for every $x\in M$ $\et(x,\delta)\neq
o(\psi(x,\delta))$ when $\delta \rightarrow 0^+$.
\end{defn}

We will denote with $\mathcal{U}_M(\et)$,
$\mathcal{U}_M^{\textrm{\,str}}(\et)$ and
$\mathcal{U}_M^{\textrm{\,opt}}(\et)$ the set of all functions $\psi
\in \mathcal{F}_M$ that are, respectively, upper bounds, strict
upper bounds and optimal upper bounds of convergence for the total
error of $\{R_\alpha\}_{\alpha \in (0,\alpha_0)}$ on $M$. In view of
Remark \ref{rem:1}, it is clear that $\et\in
\mathcal{U}_M^{\,\textrm{opt}}(\et)$ for every $M \subset X$.

The observations below follow immediately from the previous
definitions.

$\bullet$ If $\psi \in \mathcal{F}_M$, then $\psi \in
\mathcal{U}_M(\et)$ if (and only if)
$\et(x,\delta)=O(\psi(x,\delta))$ as $\delta \rightarrow 0^+$ for
every $x \in M$. Moreover, $\mathcal{U}_M^{\,\textrm{\,str}}(\et)$
and $\mathcal{U}_M^{\,\textrm{opt}}(\et)$ are disjoint subsets of
$\mathcal{U}_M(\et)$, although their union is not all of
$\mathcal{U}_M(\et)$ (except when $M$ consists of just one element).

$\bullet$ If $\tilde{M}\subset M$, then $\mathcal{U}_{M}(\et)\subset
\mathcal{U}_{\tilde{M}}(\et)$,
$\mathcal{U}_M^{\,\textrm{opt}}(\et)\subset
\mathcal{U}_{\tilde{M}}^{\,\textrm{opt}}(\et)$ and
$\mathcal{U}_M^{\,\textrm{\,str}}(\et)\subset
\mathcal{U}_{\tilde{M}}^{\,\textrm{\,str}}(\et)$.

$\bullet$ If $\psi \in \mathcal{U}_M(\et)$, $\tilde{\psi}\in
\mathcal{F}_M$ and $\psi \overset{M}{\preceq}\tilde{\psi}$, then
$\tilde{\psi} \in \mathcal{U}_M(\et).$

$\bullet$ If $\psi \in \mathcal{U}_M^{\,\textrm{opt}}(\et)$,
$\tilde{\psi}\in \mathcal{U}_M(\et) $ and $\tilde{\psi}
\overset{M}{\preceq}\psi$, then $\tilde{\psi} \in
\mathcal{U}_M^{\,\textrm{opt}}(\et).$

$\bullet$ If $\psi \in \mathcal{U}_M^{\,\textrm{\,str}}(\et)$,
$\tilde{\psi}\in \mathcal{F}_M$ and $\psi
\overset{M}{\preceq}\tilde{\psi}$, then $\tilde{\psi} \in
\mathcal{U}_M^{\,\textrm{\,str}}(\et).$

\smallskip
\begin{defn}\label{def:comparable}
Let $\psi, \tilde{\psi} \in \mathcal{F}_M$. We say that ``$\psi$
and $\tilde{\psi}$ are comparable on $M$'' if they verify $\psi
\overset{M}{\preceq}\tilde{\psi}$ or $\tilde{\psi}
\overset{M}{\preceq}\psi$ (or both).
\end{defn}

\smallskip
\begin{defn}\label{def:minimal}
Let $\mathcal{A}\subset \mathcal{F}_M$ and $\psi^\ast
\in\mathcal{A}$. We say that ``$\psi^\ast$ is a minimal element of
$\left(\mathcal{A}, \overset{M}{\preceq}\right)$'' if $\psi^\ast
\overset{M}{\preceq}\psi$ for every $\psi \in \mathcal{A}$
comparable with $\psi^\ast$ on  $M$. Equivalently, $\psi^\ast$ is
minimal element of $\left(\mathcal{A},
\overset{M}{\preceq}\right)$ if for every $\psi \in \mathcal{A}$,
the condition $\psi\overset{M}{\preceq}\psi^\ast$ implies
$\psi^\ast \overset{M}{\preceq}\psi$.
\end{defn}

\smallskip
\begin{lem}\label{lem:nomin}
Let $\mathcal{A}\subset \mathcal{F}_M, \psi,\psi^\ast \in
\mathcal{A}$ and $\psi,\psi^\ast$ be comparable on $M$. If there
exists $M_0 \subset M$ such that $\psi
\overset{M_0}{\prec}\psi^\ast$ then $\psi^\ast$ is not a minimal
element of $\left(\mathcal{A},\overset{M}{\preceq}\right).$
\end{lem}

\begin{proof}
Let $\mathcal{A}\subset \mathcal{F}_M$ and $\psi,\psi^\ast \in
\mathcal{A}$ be comparable on $M$. Let us suppose that there
exists $M_0 \subset M$ such that $\psi
\overset{M_0}{\prec}\psi^\ast$, then it follows from Lemma
\ref{lem:menor estricto} that $\psi^\ast
\overset{M_0}{\npreceq}\psi$. Thus $\psi^\ast
\overset{M}{\npreceq}\psi$ and since $\psi,\psi^\ast \in
\mathcal{A}$ are comparable on $M$, it follows from Definition
\ref{def:minimal} that $\psi^\ast$ cannot be a minimal element of
$\left(\mathcal{A},\overset{M}{\preceq}\right).$ \hfill
\end{proof}

\smallskip
\begin{cor}\label{cor:nominimal}
If $\psi^\ast \in \mathcal{U}_M(\et)$ and there exist $x_0 \in M$
and $\psi_0 \in \mathcal{U}_{\{x_0\}}(\et)$ such that
$\psi_0\overset{\{x_0\}}{\prec}\psi^\ast$ then $\psi^\ast$ is not
a minimal element of
$\left(\mathcal{U}_M(\et),\overset{M}{\preceq}\right).$
\end{cor}

\begin{proof}
This corollary is an immediate consequence of the previous lemma
with $\mathcal{A}=\mathcal{U}_M(\et)$, $M_0=\{x_0\}$ and
\begin{equation*}
\psi(x,\delta)\doteq \left\{%
\begin{array}{ll}
    \psi_0(x_0,\delta), & \hbox{if $x=x_0$} \\
    \psi^\ast(x,\delta), & \hbox{if $x \neq x_0$.} \\
\end{array}%
\right.
\end{equation*}
Note that this function $\psi$ so defined is in
$\mathcal{U}_M(\et)$ and it is comparable with $\psi^\ast$ on $M$
(moreover $\psi\overset{M}{\preceq}\psi^\ast$).\hfill
\end{proof}

Next we will show that the optimal upper bounds of convergence for
the total error of $\{R_\alpha\}_{\alpha \in (0,\alpha_0)}$ on $M$
are characterized by being minimal elements of the partially
ordered set $\left(\mathcal{U}_{M}(\et),\,
\overset{M}{\preceq}\right)$. More precisely, we have the
following result.

\begin{thm}\label{teo:minimal-optimo}
Let $\psi \in \mathcal{U}_M(\et)$. Then $\psi \in
\mathcal{U}_{M}^{\,\textrm{\rm{opt}}}(\et)$ if and only if $\psi$
is a minimal element of $\left(\mathcal{U}_M(\et),
\overset{M}{\preceq}\right)$.
\end{thm}
\begin{proof}
Let $\psi \in \mathcal{U}_{M}^{\,\textrm{opt}}(\et)$ and suppose
that $\psi$ is not a minimal element of \break
$\left(\mathcal{U}_M(\et),\, \overset{M}{\preceq}\right)$. Then
there exists $\psi_c \in \mathcal{U}_M(\et)$ comparable with
$\psi$ on $M$ for which it is not true that $\psi
\overset{M}{\preceq} \psi_c$. Then, there exists $x_0 \in M$ such
that
\begin{equation}\label{eq:5}
\underset{\delta \rightarrow
0^+}{\limsup}\;\frac{\psi(x_0,\delta)}{\psi_c(x_0,\delta)}=\infty.
\end{equation}
Now, since $\psi \in \mathcal{U}_{M}^{\,\textrm{opt}}(\et)$ and
$x_0 \in M$, we have that
\begin{equation}\label{eq:6}
\underset{\delta \rightarrow
0^+}{\limsup}\;\frac{\et(x_0,\delta)}{\psi(x_0,\delta)}>0.
\end{equation}
Thus
\begin{equation*}
\underset{\delta \rightarrow
0^+}{\limsup}\;\frac{\et(x_0,\delta)}{\psi_c(x_0,\delta)}=\underset{\delta
\rightarrow
0^+}{\limsup}\;\frac{\et(x_0,\delta)}{\psi(x_0,\delta)}\frac{\psi(x_0,\delta)}{\psi_c(x_0,\delta)}\;=\;\infty
\end{equation*}
which implies that $\psi_c \notin \mathcal{U}_{\{x_0\}}(\et)$.
This contradicts the fact that $\psi_c \in \mathcal{U}_M(\et)$.
Therefore, $\psi$ must be a minimal element of
$\left(\mathcal{U}_M(\et),\, \overset{M}{\preceq}\right)$.

Conversely, assume that $\psi \in \mathcal{U}_M(\et)$ and $\psi
\notin \mathcal{U}_M^{\,\textrm{opt}}(\et)$. Then there exists $x_0
\in M$ such that $\psi \in
\mathcal{U}_{\{x_0\}}^{\,\textrm{\,str}}(\et)$, which implies that
$\et\overset{\{x_0\}}{\prec}\psi$. Lemma \ref{lem:nomin} then
implies that $\psi$ is not a minimal element of
$\left(\mathcal{U}_M(\et),\, \overset{M}{\preceq}\right)$.\hfill
\end{proof}

From the proof of Theorem \ref{teo:minimal-optimo} it follows
immediately that $\psi$ is a minimal element of
$\left(\mathcal{U}_M(\et),\, \overset{M}{\preceq}\right)$ if and
only if it is minimal of $\left(\mathcal{U}_{M^\ast}(\et),\,
\overset{M^\ast}{\preceq}\right)$ for every $M^\ast \subset M$.
Also, as a consequence of Theorem \ref{teo:minimal-optimo} one has
that all optimal upper bounds of convergence for the total error
must be equivalent in the sense of Definition
\ref{def:relac-fm}-{\it ii}. More precisely we have the following
\begin{cor}\label{cor:um-opt} Let
$\mathcal{U}_{M}^{\,\textrm{\rm{opt}}}(\et)$ and $\et$ be as
before
\begin{itemize}
\item  [\bf{i)}] If $\psi \in \mathcal{U}_{M}^{\,\textrm{\rm{opt}}}(\et)$
then $\psi \overset{M}{\approx}\et$.

\item  [\bf{ii)}] If $\psi,\tilde{\psi} \in
\mathcal{U}_{M}^{\,\textrm{\rm{opt}}}(\et)$ then $\psi
\overset{M}{\approx}\tilde{\psi}$.
\end{itemize}

\end{cor}

\smallskip
\begin{proof}

\textbf{i)} If $\psi \in \mathcal{U}_{M}^{\,\textrm{opt}}(\et)$
then $\et\overset{M}{\preceq}\psi$, from which it follows that
$\et$ and $\psi$ are comparable on $M$. Then, since $\et\in
\mathcal{U}_M(\et)$ and  by Theorem \ref{teo:minimal-optimo}
$\psi$ is a minimal element of $\left(\mathcal{U}_M(\et),\,
\overset{M}{\preceq}\right)$, we have that
$\psi\overset{M}{\preceq}\et$. Hence, $\psi
\overset{M}{\approx}\et$.

\textbf{ii)} This is an immediate consequence of \textbf{i)} and
the transitivity and reflexivity of the equivalence relation
``$\overset{M}{\approx}$'', because by {\bf i)} every $\psi \in
\mathcal{U}_{M}^{\,\textrm{opt}}(\et)$ is equivalent to $\et$ on
$M$.\hfill
\end{proof}

This result says that if $\psi$ is an optimal upper bound of
convergence on $M$ for the total error of a regularization method,
then at every point of $M$, $\psi$ tends to zero, as the noise
level tends to zero, exactly with the ``{\it same speed}" with
which the total error does.

In order to introduce the concept of saturation in the next
section, we will previously need a few more definitions and tools
that will allow us to compare bounds of convergence on different
sets of $X$.

\begin{defn}\label{def:relac-mn}
Let $M,\tilde{M}\subset X$, $\psi \in \mathcal{F}_M$ and
$\tilde{\psi}\in \mathcal{F}_{\tilde{M}}$.

i) We say that ``$\psi$ on $M$ precedes $\tilde{\psi}$ on
$\tilde{M}$'', and we denote it with $\psi \overset{M,
\tilde{M}}{\preceq}\tilde{\psi}$, if there exist a constant $d>0$
and a function $k:M\times \tilde{M}\rightarrow (0,\infty)$ such
that $\psi(x,\delta)\leq k(x,\tilde{x})\,
\tilde{\psi}(\tilde{x},\delta)$ for every $x \in M$, for every
$\tilde{x}\in \tilde{M}$ and for every $\delta \in (0,d).$

ii) We say that ``$\psi$ on $M$ is equivalent to $\tilde{\psi}$ on
$\tilde{M}$'', and we denote it with $\psi \overset{M,
\tilde{M}}{\approx}\tilde{\psi}$, if $\,\psi \overset{M,
\tilde{M}}{\preceq}\tilde{\psi}\;$ and $\;\tilde{\psi}
\overset{\tilde{M}, M}{\preceq}{\psi}$.

iii) We say that ``$\psi$ on $M$ strictly precedes $\tilde{\psi}$
on $\tilde{M}$'', and we denote it with $\psi \overset{M,
\tilde{M}}{\prec}\tilde{\psi}$, if $\,\psi \overset{M,
\tilde{M}}{\preceq}\tilde{\psi}\,$ and $\,\underset{\delta
\rightarrow0^+}{\limsup}\,\frac{\psi(x,\delta)}{\tilde{\psi}(\tilde{x},\delta)}=0$
for every $x \in M$, $\tilde{x} \in \tilde{M}.$
\end{defn}

\begin{rem}
In a certain sense, when $M=\tilde M$, the previous definitions
generalize (although they are slightly stronger than) the
relations introduced in Definition \ref{def:relac-fm}. Note for
instance that if $\psi \overset{M, M}{\prec}\tilde{\psi}$ then
$\psi \overset{M}{\prec}\tilde{\psi}$, although the converse, in
general, is not true.
\end{rem}

It follows immediately from Definition \ref{def:relac-mn} that if
$\psi\overset{M,N}{\preceq}\,\tilde{\psi}$ then
$\psi\overset{\tilde{M},\tilde{N}}{\preceq}\,\tilde{\psi}$ for
every $\tilde{M}\subset M$ and for every $\tilde{N}\subset N.$ The
same happens for the relations ``$\overset{M, N}{\approx}$'' and
``$\overset{M, N}{\prec}$''.

Next, we also need to extend the notion of ``comparability'' given
in Definition \ref{def:comparable}, to this case.

\begin{defn}
Let $M,\tilde{M}\subset X$, $\psi \in \mathcal{F}_M$ and
$\tilde{\psi}\in \mathcal{F}_{\tilde{M}}$.

i) We say that ``$\psi$ on $M$ is comparable with $\tilde{\psi}$
on $\tilde{M}$'' if $\psi
\overset{M,\tilde{M}}{\preceq}\tilde{\psi}$ or $\tilde{\psi}
\overset{\tilde{M},M}{\preceq}\psi.$

ii) We say that ``$\psi$ is invariant over $M$'' if $\psi
\overset{M,M}{\approx}\psi$.
\end{defn}

\begin{rem}\label{rem:MM}
It is immediate that the condition $\psi
\overset{M,M}{\approx}\psi$ is equivalent to $\psi
\overset{M,M}{\preceq}\psi$.
\end{rem}

This last notion of ``invariance'', which will play an important
roll in the characterization of saturation, roughly speaking
establishes that if $\psi$ is invariant over $M$ then the orders
of convergence of $\psi$ as a function of $\delta$ when $\delta
\rightarrow 0^+$, in any two points of $M$, are equivalent.

The following result is related to a certain transitivity property
of this invariance relation.

\begin{lem}\label{lem:transit}
Let $M\subset X$, $\psi,\tilde{\psi}\in \mathcal{F}_M$ be such
that $\tilde{\psi}\overset{M}{\approx}\psi$ and
$\psi\overset{M,M}{\approx}\psi$. Then:
\begin{itemize}
\item [\bf{i)}] $\tilde{\psi}\overset{M,M}{\approx}\psi$ and
\item [\bf{ii)}] $\tilde\psi\overset{M,M}{\approx}\tilde\psi$ (i.e.
$\tilde{\psi}$ is also invariant over $M$).
\end{itemize}
\end{lem}
\begin{proof}

Let $M\subset X$, $\psi,\tilde{\psi}\in \mathcal{F}_M$,
$x,\tilde{x}\in M$ and suppose that
$\psi\overset{M,M}{\approx}\psi$ and
$\tilde{\psi}\overset{M}{\approx}\psi$.

\textbf{i)} Since $\tilde{\psi}\overset{M}{\approx}\psi$, there
exist positive constants $d,k_x$ and $k_{\tilde{x}}$ such that for
every $\delta \in (0,d)$,
\begin{equation}\label{eq:7}
\tilde{\psi}(x,\delta)\leq k_x\psi(x,\delta) \quad
\textrm{and}\quad \psi(\tilde{x},\delta)\leq
k_{\tilde{x}}\tilde{\psi}(\tilde{x},\delta).
\end{equation}
On the other hand, from the invariance of $\psi$ over $M$ it
follows that there exist positive constants $d^\ast$ y
$k^\ast_{x,\tilde{x}}$ such that $\psi(x,\delta)\leq
k^\ast_{x,\tilde{x}}\psi(\tilde{x},\delta)$ for every $\delta \in
(0,d^\ast)$, which together with (\ref{eq:7}) implies that for
every  $\delta \in (0,\min\{d,d^\ast\})$,
\begin{equation}\label{eq:10}
 \tilde{\psi}(x,\delta)\leq k_x\psi(x,\delta)\leq
  k_xk^\ast_{x,\tilde{x}}\psi(\tilde{x},\delta) \quad
  \textrm{and}\quad \psi(x,\delta)\leq k^\ast_{x,\tilde{x}}
\psi(\tilde{x},\delta) \leq k^\ast_{x,\tilde{x}}
  k_{\tilde{x}}\tilde{\psi}(\tilde x,\delta).
\end{equation}
Since $x,\tilde x\in M$ are arbitrary, it follows that
$\tilde{\psi}\overset{M,M}{\preceq}\psi$ and
$\psi\overset{M,M}{\preceq}\tilde{\psi}$, that is,
$\tilde{\psi}\overset{M,M}{\approx}\psi$.

\textbf{ii)} From the first inequality in (\ref{eq:10}) and from
the second inequality in (\ref{eq:7}) it follows immediately that
$\tilde{\psi}\overset{M,M}{\preceq}\tilde{\psi}$ and therefore by
Remark \ref{rem:MM}, $\tilde{\psi}$ is invariant over $M$.\hfill
\end{proof}

The following result is analogous to Lemma \ref{lem:menor
estricto} for this case of comparison of convergence bounds on
different sets.

\begin{lem}\label{lem:prec-2}
Let $M,N\subset X$, $\psi \in \mathcal{F}_M$ and $\tilde{\psi}\in
\mathcal{F}_{N}$. If $\psi \overset{M, N}{\prec}\tilde{\psi}$ then
$\forall \;\tilde{M}\subset M$, $\forall \; \tilde{N}\subset N$ we
have that $\tilde{\psi} \overset{\tilde{N},
\tilde{M}}{\npreceq}\psi$.
\end{lem}

\begin{proof}
By the contrareciprocal. Suppose that there exist $\tilde{M}
\subset M$ and $\tilde{N} \subset N$ such that
$\tilde{\psi}\overset{\tilde{N}, \tilde{M}}{\preceq}\psi$. Then
there exist a constant $d>0$ and $k:\tilde{N}\times
\tilde{M}\rightarrow (0,\infty)$ such that
$\tilde{\psi}(\tilde{x},\delta)\leq k(\tilde{x},x)\,
\psi(x,\delta)$ for every $\tilde{x} \in \tilde{N}$, $x \in
\tilde{M}$ and $\delta \in (0,d).$ Let $x_0 \in \tilde{M}$ and
$\tilde{x}_0 \in \tilde{N}$, then
$\tilde{\psi}(\tilde{x}_0,\delta)\leq k(\tilde{x}_0,x_0)\,
\psi(x_0,\delta)$ for every $\delta \in (0,d).$ Thus,
$\underset{\delta \,\in
(0,d)}{\sup}\frac{\tilde{\psi}(\tilde{x}_0,\delta)}{\psi(x_0,\delta)}\leq
k(\tilde{x}_0,x_0)<\infty$. Then,
\begin{eqnarray*}
\underset{\delta\rightarrow0^+}{\limsup}\frac{\psi(x_0,\delta)}{\tilde{\psi}(\tilde{x}_0,\delta)}&\geq&
\underset{\delta\rightarrow0^+}{\liminf}\frac{\psi(x_0,\delta)}{\tilde{\psi}(\tilde{x}_0,\delta)}\geq
\underset{\delta\in
(0,d)}{\inf}\frac{\psi(x_0,\delta)}{\tilde{\psi}(\tilde{x}_0,\delta)}\\
  &=&\left(\underset{\delta\in
(0,d)}{\sup}\frac{\tilde{\psi}(\tilde{x}_0,\delta)}{\psi(x_0,\delta)}\right)^{-1}\geq\frac{1}{k(\tilde{x}_0,x_0)}>0.
\end{eqnarray*}

Hence, $\psi\overset{\{x_0\},
\{\tilde{x}_0\}}{\nprec}\tilde{\psi}$, from which it follows that
$\psi\overset{M, N}{\nprec}\tilde{\psi}$, since $x_0\in M$ and
$\tilde x_0\in N$ .\hfill
\end{proof}

\section{Global Saturation} \label{sec:3}

We will now proceed to formalize the concept of global saturation.

\begin{defn}\label{def:satur}
Let $M_S \subset X$ and $\psi_S \in \mathcal{U}_{M_S}(\et)$. We
say that $\psi_S$ is a ``global saturation function of
$\{R_\alpha\}$ over $M_S$'' if $\psi_S$ satisfies the following
three conditions:

S1. For every $x^\ast \in X$, $x^\ast\neq 0$, $x \in M_S$,
$\underset{\delta
\rightarrow0^+}{\limsup}\,\frac{\et(x^\ast,\delta)}{\psi_S(x,\delta)}>0.$

S2. $\psi_S$ is invariant over $M_S$.

S3. There is no upper bound of convergence for the total error of
$\{R_\alpha\}$ that is a proper extension of $\psi_S$ (in the
variable $x$) and satisfies S1 and S2, that is, there exist no
$\tilde{M}\supsetneqq M_S$ and $\tilde{\psi} \in \mathcal{U}_{
\tilde{M}}(\et)$ such that $\tilde{\psi}$ satisfies S1 and S2 with
$M_S$ replaced by $\tilde{M}$ and $\psi_S$ replaced by
$\tilde{\psi}$.
\end{defn}

We shall refer to  $\psi_S$ and $M_S$ as the saturation function
and the saturation set, respectively.

\begin{rem}\label{obs:sat}
Note that condition S1 implies that for every $M\subset X$ and for
every $\psi \in \mathcal{U}_M (\et)$, $\underset{\delta
\rightarrow0^+}{\limsup}\,\frac{\psi(x^\ast,\delta)}{\psi_S(x,\delta)}>0$
for every $x^\ast \in M$, $x \in M_S$ (this is an immediate
consequence of S1 and the fact that $\et\overset{M}{\preceq}\psi
\; \forall \; \psi \in \mathcal{U}_M(\et)$). Therefore, it cannot
happen that $\psi\overset{M, M_S}{\prec}\psi_S$. On the other
hand, if $\psi \in \mathcal{U}_M (\et)$ then it is not necessarily
true that $\psi_S\overset{M_S, M}{\preceq}\psi$ even if $\psi$ on
$M$ is comparable to $\psi_S$ on $M_S$, because in this case it
can happen that $\underset{\delta
\rightarrow0^+}{\liminf}\,\frac{\psi(x,\delta)}{\psi_S(x_S,\delta)}=0$
for some $x \in M$ and some $x_S \in M_S$ (which obviously implies
that $\psi_S\overset{M_S, M}{\npreceq}\psi$), and still have
$\underset{\delta
\rightarrow0^+}{\limsup}\,\frac{\psi(x,\delta)}{\psi_S(x_S,\delta)}>0$.
However, if $\psi$ on $M$ is comparable with $\psi_S$ on $M_S$ and
there exists $\underset{\delta
\rightarrow0^+}{\lim}\,\frac{\psi(x,\delta)}{\psi_S(x_S,\delta)}$
for every $x \in M$ and for every $x_S \in M_S$, then it is in
fact true that $\psi_S\overset{M_S, M}{\preceq}\psi$. Note also
that condition S1 can be replaced by
$$\underset{\delta
\rightarrow0^+}{\limsup}\,\frac{\psi(x^\ast,\delta)}{\psi_S(x,\delta)}>0\quad
\forall \; \psi \in \mathcal{U}_{\{x^\ast\}}(\et), \forall\;
x^\ast \in X, x^\ast \neq 0, x\in M_S.
$$
\end{rem}

This conception of global saturation essentially establishes that
in no point $x^\ast \in X$, $x^\ast \neq 0$, can exist an upper
bound of convergence for the total error of the regularization
method that is ``strictly better" than the saturation function
$\psi_S$ at any point of the saturation set $M_S$.

Next we show that any function satisfying condition S1, in
particular any saturation function, is always an optimal upper
bound of convergence.

\smallskip
\begin{lem}\label{lem:sat-opt}
Let $\psi_S \in \mathcal{U}_{M_S}(\et)$. If $\psi_S$ satisfies the
condition \textit{S1} on $M_S$, then $\psi_S \in
\mathcal{U}_{M_S}^{\,\textrm{\rm{opt}}}(\et)$.
\end{lem}
\begin{proof}
The condition \textit{S1} implies in particular that
$\underset{\delta
\rightarrow0^+}{\limsup}\,\frac{\et(x,\delta)}{\psi_S(x,\delta)}>0$
for every $x \in M_S$. Since also by definition $\psi_S \in
\mathcal{U}_{M_S}(\et)$ it follows that $\psi_S$ is an optimal
upper bound of convergence for the total error of $\{R_\alpha\}$,
i.e. $\psi_S \in \mathcal{U}_{M_S}^{\,\textrm{opt}}(\et)$. \hfill
\end{proof}

An immediate consequence of this lemma is the equivalence between
the saturation function and the total error on the saturation set.

\begin{cor}\label{cor:sat-equiverror}
If $\psi_S$ is a saturation function of $\{R_\alpha\}$ on $M_S$
then $\psi_S \overset{M_S}{\approx}\et$. Moreover, we have the
stronger equivalence $\psi_S \overset{M_S,M_S}{\approx}\et$.
\end{cor}

\smallskip
\begin{proof}
The first part of the corollary is an immediate consequence of the
previous lemma and of Corollary \ref{cor:um-opt} \textbf{i)}. The
second part follows from the first and the fact that $\psi_S
\overset{M_S,M_S}{\approx}\psi_S$, via Lemma \ref{lem:transit}
\textbf{i)}.\hfill
\end{proof}

\smallskip
\begin{rem}\label{obs:invariancia-et}
A consequence of the first part of this corollary and of Lemma
\ref{lem:transit} ii) is that if $\psi_S$ is a saturation function
of  $\{R_\alpha\}$ on $M_S$, then $\et
\overset{M_S,M_S}{\approx}\et$, that is, the total error must be
invariant over $M_S$. We will shed more light on this matter in
Theorem \ref{teo:caract-sat}.
\end{rem}

\smallskip
\begin{defn}\label{def:inmejorable}
Let $M\subset X$ and $\psi \in \mathcal{U}_{\,X}(\et)$. We say
that ``$M$ is optimal for $\psi$'', and we denote it with $M \in
\mathcal{O}(\psi)$, if the following condition holds:

C2. For every $x \in M$, $x_c \in M^c$ neither $\psi
\overset{\{x_c\}, \{x\}}{\prec}\psi$ nor $\psi \overset{\{x_c\},
\{x\}}{\approx}\psi$.
\end{defn}

\smallskip
That a set $M$ be optimal for $\psi$ essentially means that at any
point of the complement of $M$, the order of convergence of $\psi$
as a function of $\delta$, for $\delta \rightarrow 0^+$, cannot be
better nor even equivalent to the order of convergence of $\psi$
at any point outside $M$; that is, at any point outside of $M$,
the order of convergence of $\psi$ must be strictly worse than
itself at any point of $M$. However, we will see next that this
optimality condition imposes a very precise restriction. As we
shall see later on (Theorem \ref{teo:caract-sat}), it is precisely
this property of the total error, together with its invariance on
the set $M_S$, what will allow us to characterize the
regularization methods which do have saturation.

Condition {\it C2} is very precise and gives no room for maneuver.
In fact, let $\psi \in \mathcal{U}_{\,X}(\et)$, $M \subset X$ and
consider the following conditions:

\textit{C1}. $\quad \psi\overset{M,M^c}{\prec}\psi.$

\textit{C3}. $\quad \psi \overset{M^c,M}{\nprec}\psi\;$ and
$\;\psi \overset{M^c,M}{\napprox}\psi$.

Then it follows that condition \textit{C2} (of optimal set) is
strictly stronger than condition \textit{C3}, and strictly weaker
than condition \textit{C1}. In fact, if $M$ is optimal for $\psi$
in the sense of Definition \ref{def:inmejorable}, then for every
$x\in M$, $x_c \in M^c$ we have that $\psi \overset{\{x_c\},
\{x\}}{\nprec}\psi$ and $\psi \overset{\{x_c\},
\{x\}}{\napprox}\psi$, from which it follows immediately that
$\psi \overset{M^c,M}{\nprec}\psi$ and $\psi
\overset{M^c,M}{\napprox}\psi$, that is, \textit{C3} holds.
However, for condition \textit{C3} to hold it is sufficient that
there exist $x\in M$ and $x_c \in M^c$ such that $\psi
\overset{\{x_c\}, \{x\}}{\nprec}\psi$ and $\psi \overset{\{x_c\},
\{x\}}{\napprox}\psi$, which obviously does not imply condition
\textit{C2}. On the other hand if \textit{C1} holds, then it
follows from Lemma \ref{lem:prec-2} that for every $x \in M$, $x_c
\in M^c$, there holds $\psi \overset{\{x_c\},\{x\}}{\npreceq}\psi$
and therefore, $\psi \overset{\{x_c\}, \{x\}}{\nprec}\psi$ and
$\psi \overset{\{x_c\}, \{x\}}{\napprox}\psi$ for every $x \in M$,
$x_c \in M^c$, that is, condition \textit{C2} holds. However,
\textit{C2} does not imply \textit{C1} since it can happen that
$M$ be optimal for $\psi$ and that there exist $x \in M$ and $x_c
\in M^c$ such that $\psi$ on $\{x\}$ is not comparable with $\psi$
on $\{x_c\}$. This implies in particular that $\psi
\overset{\{x\},\{x_c\}}{\npreceq}\psi$ and therefore, $\psi
\overset{M,M^c}{\nprec}\psi$.

In order to be able to characterize the regularization methods
which do have saturation, we will previously need the following
result.

\smallskip
\begin{lem}\label{lem:noequiv}
Suppose that $\{R_\alpha\}$ has saturation function on $M\subset
X$ and for every $x \in M$, $x_c \in M^c$ there holds $\et
\overset{\{x_c\}, \{x\}}{\nprec}\et$. Then $\et \overset{\{x_c\},
\{x\}}{\napprox}\et$ for every $x \in M$, $x_c \in M^c$.
\end{lem}

\begin{proof}
Since $\{R_\alpha\}$ has saturation function on $M$, it follows
from Remark \ref{obs:invariancia-et} that $\et$ is invariant over
$M$. Suppose that $\et \overset{\{x_c\}, \{x\}}{\nprec}\et$ for
every $x \in M$, $x_c \in M^c$ and that there exist $\tilde{x} \in
M$, $\tilde{x}_c \in M^c$ such that
\begin{equation}\label{eq:3}
\et \overset{\{\tilde{x}_c\}, \{\tilde{x}\}}{\approx}\et.
\end{equation}
Then,
\begin{equation}\label{eq:4}
\underset{\delta
\rightarrow0^+}{\limsup}\,\frac{\et(\tilde{x},\delta)}{\et(\tilde{x}_c,\delta)}>0.
\end{equation}
Define $\tilde{M}\doteq M \cup \{\tilde{x}_c\}$ and
\begin{equation*}
\tilde{\psi}(x,\delta)\doteq \left\{%
\begin{array}{ll}
    \psi(x,\delta), & \hbox{if $x\in M$} \\
    \et(x,\delta), & \hbox{if $x=\tilde{x}_c $,} \\
\end{array}%
\right.
\end{equation*}
where $\psi$ is a saturation function of $\{R_\alpha\}$ on $M$. We
will show next that $\tilde{\psi}$ is saturation function on
$\tilde{M}$. Clearly, $\tilde{\psi}$ is upper bound of convergence
for the total error on $\tilde{M}$, i.e., $\tilde{\psi}\in
\mathcal{U}_{\tilde{M}}(\et)$ and since $\psi$ is saturation on
$M$, it follows that $\tilde{\psi}(x,\delta)$ satisfies condition
\textit{S1} for all $x\in M$. We will now check that
$\tilde{\psi}(\tilde{x}_c,\delta)$ also satisfies \textit{S1}.
Since $\tilde{x}\in M$ it follows that
\begin{equation} \label{eq:8}
\underset{\delta
\rightarrow0^+}{\limsup}\,\frac{\et(x^\ast,\delta)}{\et(\tilde{x},\delta)}>0\quad
\forall \;x^\ast \in X,\, x^\ast \neq 0.
\end{equation}
If $x^\ast \in M$, the above inequality follows from the fact that
$\et$ is invariant over $M$ and if $x^\ast \in M^c$, it is a
consequence of the fact that $\et \overset{\{x^\ast\},
\{\tilde{x}\}}{\nprec}\et$.

Then, for every  $x^\ast \in X$, $x^\ast \neq0$ we have that
\begin{equation*}
\underset{\delta \rightarrow
0^+}{\limsup}\;\frac{\et(x^\ast,\delta)}{\tilde{\psi}(\tilde{x}_c,\delta)}=\underset{\delta
\rightarrow
0^+}{\limsup}\;\frac{\et(x^\ast,\delta)}{\et(\tilde{x},\delta)}\;\frac{\et(\tilde{x},\delta)}{\tilde{\psi}
(\tilde{x}_c,\delta)}>0
\end{equation*}
by virtue of (\ref{eq:4}) and (\ref{eq:8}). Thus,
$\tilde{\psi}(x,\delta)$ satisfies \textit{S1} for every $x\in
\tilde{M}.$

We will now check that $\tilde{\psi}$ satisfies \textit{S2} on
$\tilde{M}$. Since $\psi$ is saturation function of $\{R_\alpha\}$
on $M$, and $\tilde\psi|_M=\psi$ we have that $\tilde{\psi}$ is
invariant over $M$. It remains to prove that $\tilde{\psi}
\overset{\{\tilde{x}_c\}, M}{\approx}\tilde{\psi}$, i.e. that $\et
\overset{\{\tilde{x}_c\}, M}{\approx}\psi$. But this is an
immediate consequence of (\ref{eq:3}), of Corollary
\ref{cor:sat-equiverror} which implies that $\psi
\overset{M}{\approx}\et$ and the fact that $\psi$ is invariant
over $M$.

Thus, we have shown that $\tilde{\psi}$ is a proper extension of
$\psi$ satisfying \textit{S1} and \textit{S2} on $\tilde{M}$,
which then implies that $\psi$ does not satisfy condition
\textit{S3}. This contradicts the fact that $\psi$ is saturation
function of $\{R_\alpha\}$ on $M$. Therefore, for every $x\in M$,
$x_c \in M^c$ there must hold that $\et \overset{\{x_c\},
\{x\}}{\napprox}\et$.\hfill
\end{proof}

\smallskip
\begin{thm}\label{teo:caract-sat} (Necessary and sufficient condition
for the existence of saturation.) A regularization method
$\{R_{\alpha}\}$ has saturation function if and only if there
exists $M \subset X$ $(M\neq \{0\}, M\neq\emptyset)$ such that
$\et$ is invariant over $M$ and $M$ is optimal for $\et$. In this
case
$\mathcal{E}^{\text{\rm{tot}}}_M(x,\delta)\doteq\et(x,\delta)$ for
$x \in M$ and $\delta>0$ is saturation function of
$\{R_{\alpha}\}$ on $M$.
\end{thm}

\smallskip
\begin{proof}
Suppose that $\{R_{\alpha}\}$ has saturation function $\psi$ on
$M$. Then it follows from Remark \ref{obs:invariancia-et} that
$\et$ is invariant over $M$.

Let us now check that $M$ is optimal for $\et$. Let $x \in M$ and
$x_c \in M^c$. We will first show that  $\et \overset{\{x_c\},
\{x\}}{\nprec}\et$. Since $\psi \in \mathcal{U}_M(\et)$ and $x\in
M$, there exist positive constants $d$ and $k_x$ such that
$\et(x,\delta) \leq k_x \psi(x,\delta)$ for every $\delta \in
(0,d).$ Then
\begin{equation*}
\underset{\delta
\rightarrow0^+}{\limsup}\,\frac{\et(x_c,\delta)}{\et(x,\delta)}\geq
\underset{\delta
\rightarrow0^+}{\limsup}\,\frac{\et(x_c,\delta)}{k_x
\psi(x,\delta)}>0,
\end{equation*}
where the last inequality follows from the fact that $\psi$
satisfies condition \textit{S1} on $M$. Therefore $\forall \;x\in
M$, $\forall \;x_c \in M^c$, $\et \overset{\{x_c\},
\{x\}}{\nprec}\et$. This condition together with the fact that
$\et$ is invariant over $M$ implies, by virtue of Lemma
\ref{lem:noequiv}, that $\forall \;x\in M$, $\forall \;x_c \in
M^c$, $\et \overset{\{x_c\}, \{x\}}{\napprox}\et$. We have thus
shown that $M$ is optimal for $\et$.

Conversely, suppose that there exists $M \subset X$ $(M\neq \{0\},
M\neq\emptyset)$ such that $\et$ is invariant over $M$ and $M$ is
optimal for $\et$ and define
$\mathcal{E}^{\text{tot}}_M(x,\delta)\doteq\et(x,\delta)$ for $x
\in M$ and $\delta>0$. We will show that
$\mathcal{E}^{\text{tot}}_M$ is saturation function of
$\{R_{\alpha}\}$ on $M$. Clearly, $\mathcal{E}^{\text{tot}}_M \in
\mathcal{U}_M(\et)$ and since by hypothesis
$\mathcal{E}^{\text{tot}}_M$ is invariant over $M$, it only
remains to be shown that $\mathcal{E}^{\text{tot}}_M$ satisfies
conditions \textit{S1} and \textit{S3}.

In order to prove \textit{S1}, let $x^\ast \in X$, $x^\ast \neq 0$
and $x\in M.$ If $x^\ast \in M$, then the invariance of
$\mathcal{E}^{\text{tot}}_M$ over $M$ implies that
$\mathcal{E}^{\text{tot}}_M \overset{\{x^\ast\},\{x\}}{\approx}
\mathcal{E}^{\text{tot}}_M$ and therefore
\begin{equation}\label{eq:9}
\underset{\delta
\rightarrow0^+}{\limsup}\,\frac{\et(x^\ast,\delta)}{\mathcal{E}^{\text{tot}}_M(x,\delta)}=\underset{\delta
\rightarrow0^+}{\limsup}\,\frac{\mathcal{E}^{\text{tot}}_M(x^\ast,\delta)}{\mathcal{E}^{\text{tot}}_M(x,\delta)}>0.
\end{equation}
On the other hand, if $x^\ast \in M^c$, the previous limit is also
positive due to the fact that $\et
\overset{\{x^\ast\},\{x\}}{\nprec}\mathcal{E}^{\text{tot}}_M$ (by
condition \textit{C2}) because $M$ is optimal for $\et$. Then,
$\mathcal{E}^{\text{tot}}_M$ satisfies condition \textit{S1}.

Finally, suppose that $\mathcal{E}^{\text{tot}}_M$ does not
satisfy condition \textit{S3}, i.e. there exist
$\tilde{M}\supsetneqq M$ and $\tilde{\psi} \in
\mathcal{U}_{\tilde{M}}(\et)$ such that $\tilde{\psi}$ is a proper
extension of $\mathcal{E}^{\text{tot}}_M$ satisfying conditions
\textit{S1} and \textit{S2} on $\tilde{M}$. Let $\tilde{x}\in
\tilde{M} \setminus M$, then the invariance of $\tilde{\psi}$ over
$\tilde{M}$ implies that $\tilde{\psi}
\overset{\{\tilde{x}\},M}{\approx}\tilde{\psi}$ and since
$\tilde{\psi}$ coincides with $\mathcal{E}^{\text{tot}}_M$ on $M$,
it follows that
\begin{equation}\label{eq:15}
\tilde{\psi}
\overset{\{\tilde{x}\},M}{\approx}\mathcal{E}^{\text{tot}}_M.
\end{equation}
Now since $\tilde{\psi} \in \mathcal{U}_{\tilde{M}}(\et)$
satisfies \textit{S1} on $\tilde{M}$, Lemma \ref{lem:sat-opt}
implies that $\tilde{\psi}\in
\mathcal{U}_{\tilde{M}}^{\,opt}(\et)$. Then, by virtue of
Corollary \ref{cor:um-opt}.\textit{i} we have that $\et
\overset{\tilde{M}}{\approx}\tilde{\psi}$. In particular, $\et
\overset{\{\tilde{x}\}}{\approx}\tilde{\psi}$, which, together
with (\ref{eq:15}) imply that $\et
\overset{\{\tilde{x}\},M}{\approx}\mathcal{E}^{\text{tot}}_M$,
that is, $\et \overset{\{\tilde{x}\},M}{\approx}\et$. But since
$\tilde{x} \in M^c$, this equivalence contradicts the fact that
$M$ is optimal for $\et$. Therefore, $\mathcal{E}^{\text{tot}}_M$
must satisfy condition \textit{S3} and, as a consequence, it is
saturation function of $\{R_{\alpha}\}$ on $M$. \hfill
\end{proof}

\smallskip
\begin{rem}
From the previous theorem we conclude that a saturation function
of a regularization method is an optimal upper bound of
convergence for the total error, invariant and without proper
extensions.

Note that a saturation function must be optimal in two senses. In
fact, if $\psi$ is saturation function on $M$, then $M$ is optimal
for $\psi$ and $\psi$ is optimal (upper bound) for the total error
of $\{R_\alpha\}$ on $M$. Moreover, $M$ and $\psi$ (modulus $M,M$
equivalence) are uniquely determined. In fact, if the domain $M$
is changed, then $M$ is no longer optimal for $\psi$ and if the
function $\psi$ is changed, even at a single point of $M$, in such
a way that $\psi$ is not invariant on $M$, then $\psi$ it is no
longer an optimal upper bound. Suppose that at a point $x_0\in M$,
we redefine $\psi$ as $\tilde{\psi}(x_0,\delta)$, where
$\tilde{\psi}\in \mathcal{F}_M$. If
$\tilde{\psi}\overset{\{x_0\}}{\prec}\psi$, then $\psi$ is no
longer an upper bound for the total error of $\{R_\alpha\}$ on $M$
and if $\psi\overset{\{x_0\}}{\prec}\tilde{\psi}$ then $\psi$ is
upper bound but it is not optimal. Thus for every $\tilde M\subset
M$ and for every $\tilde\psi\in \mathcal{F}_{\tilde M}$, if $\psi$
and $\tilde\psi$ are comparable on $\tilde M$ then $\psi
\overset{\tilde M,\tilde M}{\approx}\tilde{\psi}$ must hold.
\end{rem}

\section{Saturation for Spectral Regularization Methods} \label{sec:4}
The objective of this section is to apply the theory previously
developed to the case of spectral regularization methods. Further,
we show that this theory is consistent with previously existing
results about optimal convergence of spectral regularization
methods.

Let $\{E_\lambda\}_{\lambda \in \R}$ be the spectral family
associated to the linear selfadjoint operator $T^\ast T$ and
$\{g_\alpha\}_{\alpha \in (0, \alpha_0)}$ a parametric family of
functions  $g_\alpha:[0,\norm{T}^2]\rightarrow \R$ for $\alpha \in
(0,\alpha_0)$, and consider the following standing hypotheses:

\textit{H1}.\; For every $\alpha \in (0,\alpha_0)$ the function
$g_\alpha$ is piecewise continuous on $[0,\norm{T}^2]$.

\textit{H2}.\; There exists a constant $C>0$ (independent of
$\alpha$) such that $\abs{\lambda g_\alpha (\lambda)}\leq C$ for
every $\lambda \in [0,\norm{T}^2]$.

\textit{H3}.\; For every  $\lambda \in (0,\norm{T}^2]$, there
exists $\underset{\alpha \rightarrow 0^+}{\lim}
g_\alpha(\lambda)=\frac{1}{\lambda}$.

\textit{H4}.\; $G_\alpha \doteq
\norm{g_\alpha(\cdot)}_\infty=O\left(\frac{1}{\sqrt{\alpha}}\right)$
for $\alpha \rightarrow 0^+$.
\smallskip

If $\{g_\alpha\}_{\alpha \in (0, \alpha_0)}$ satisfies hypotheses
\textit{H1-H3}, then (see \cite{bookEHN}, Theorem 4.1) the
collection of operators $\{R_\alpha\}_{\alpha \in (0, \alpha_0)}$,
where
\begin{equation}\label{eq:Ralfa-galfa}
    R_\alpha\doteq \int g_\alpha(\lambda)\, dE_\lambda \,T^\ast
    =g_\alpha(T^\ast T)T^\ast,
\end{equation}
is a family of regularization operators for $T^\dag$. In this case
we say that $\{R_\alpha\}_{\alpha \in (0, \alpha_0)}$ is a family
of spectral regularization operators for $Tx=y$.

Next, we recall the classical definition of qualification for a
family of spectral regularization operators.

\smallskip
\begin{defn}
Let $\{R_\alpha\}_{\alpha \in (0,\alpha_0)}$ be the family of
spectral regularization operators for $Tx=y$ generated by the
family of functions $\{g_\alpha\}_{\alpha \in (0,\alpha_0)}$,
$r_\alpha(\lambda)\doteq 1-\lambda g_\alpha(\lambda)$,
$0<\alpha<\alpha_0$, $0\le\lambda\le \|T\|^2$, and let us denote
with $\mathcal{I}(g_\alpha)$ the set
$$\mathcal{I}(g_\alpha)\doteq\{\mu\geq 0:\exists\; k>0
\textrm{ and }\lambda^\mu\abs{r_\alpha(\lambda)}\leq
k\,\alpha^\mu\; \forall\; \lambda \in [0,\norm{T}^2],\; \forall\,
\alpha \in (0,\alpha_0)\}.
$$

The order of the classical qualification of $\{R_\alpha\}_{\alpha
\in (0,\alpha_0)}$ is defined to be $\mu_0 \doteq \underset{\mu
\in \mathcal{I}(g_\alpha)}{\sup}\, \mu$ and we say that
$\{R_\alpha\}_{\alpha \in (0,\alpha_0)}$ has classical
qualification of order $\mu_0.$
\end{defn}

\smallskip
\begin{rem}
Note that by virtue of \textit{H2}, $0 \in \mathcal{I}(g_\alpha)$
and the order $\mu_0$ of the classical qualification of a
regularization method is always nonnegative (it can be equal to  0
or $+ \infty$).
\end{rem}

\subsection{Spectral Methods with Classical Qualification of Finite Positive Order}

We start by considering first the case of spectral methods for
which $0<\mu_0<\infty$. For these methods we will first show the
existence of certain upper bounds of convergence and then we will
show that they saturate. We will also characterize their
saturation functions and saturation sets.

\begin{lem}\label{lem:cotasup-xmu}
Suppose that $\{g_\alpha\}_{\alpha\in (0,\alpha_0)}$ satisfies the
hypotheses \textit{H1-H4}. If the family of regularization
operators $\{R_\alpha\}_{\alpha\in (0,\alpha_0)}$, with $R_\alpha$
defined as in (\ref{eq:Ralfa-galfa}), has classical qualification
of order $\mu_0$, $0<\mu_0<+\infty$, then
$\psi_{\mu_0}(x,\delta)\doteq \delta^{\frac{2\mu_0}{2\mu_0+1}}$,
for $x \in X_{\mu_0}\doteq \mathcal{R}((T^\ast
T)^{\mu_0})\setminus \{0\}$ and $\delta
>0$, is upper bound of convergence for the total error of
$\{R_\alpha\}_{\alpha\in (0,\alpha_0)}$ on $X_{\mu_0}$, that is,
$\psi_{\mu_0} \in \mathcal{U}_{X_{\mu_0}}(\et)$.
\end{lem}

\begin{proof}
Since $\{g_\alpha\}$ satisfies hypothesis \textit{H4}, we have
that $G_\alpha =O\left(\frac{1}{\sqrt{\alpha}}\right)$ when
$\alpha \rightarrow 0^+$ and therefore $G_\alpha
=o(\frac{1}{\alpha})$ when $\alpha \rightarrow 0^+.$ From this and
from the fact that $\{g_\alpha\}$ satisfies hypothesis
\textit{H1-H3} and $\{R_\alpha\}$ has classical qualification of
order $\mu_0$, $0<\mu_0<+\infty$, it follows that (see
\cite{bookEHN}, Corollary 4.4 and Remark 4.5 therein) there exists
an \textit{a-priori} parameter choice rule
$\alpha^\ast:\R^+\rightarrow (0,\alpha_0)$ such that the
regularization method  $(R_{\alpha},\alpha^\ast)$ is of optimal
order on $X_{\mu_{0}}$, that is, for every $x \in X_{\mu_0}$ there
exists $k(x)>0$ such that for every $\delta>0$,
\begin{equation*}
\underset{y^\delta \in \overline{B_\delta(Tx)}}{\sup}
\norm{R_{\alpha^\ast(\delta)} y^\delta - x} \leq k(x)
\,\delta^{\frac{2\mu_{0}}{2\mu_{0}+1}}.
\end{equation*}
Then
\begin{equation*}
    \underset{\alpha \in (0,\alpha_0)}{\inf}\; \underset{y^\delta \in
\overline{B_\delta(Tx)}}{\sup} \norm{R_\alpha y^\delta -
    x} \leq k(x) \,\delta^{\frac{2\mu_{0}}{2\mu_{0}+1}},
\end{equation*}
that is, $\et(x,\delta)\le k(x)\psi_{\mu_0}(\delta)$. Thus $\et
\overset{X_{\mu_0}}{\preceq}\psi_{\mu_0}$ and therefore
$\psi_{\mu_0} \in \mathcal{U}_{X_{\mu_0}}(\et)$. $ $
\hphantom{xxxxx} \hfill
\end{proof}

\smallskip
\begin{thm}\label{teo:sat-espectral}(Saturation for families of spectral
regularization operators with classical qualification of finite
positive order.)

Suppose that $\{g_\alpha\}_{\alpha>0}$ satisfies hypotheses
\textit{H1-H4} and let $r_\alpha(\lambda)\doteq 1-\lambda
g_\alpha(\lambda)$. Suppose further that:

\textbf{i)} The spectrum of $T^\ast T$ has $\lambda=0$ as
accumulation point.

 \textbf{ii)} There exist positive constants  $\gamma_1,\gamma_2,\lambda_1, c_1$, with $\lambda_1 \leq
\norm{T}^2$ and $c_1>1$ such that

\quad \textbf{a)} $0 \leq r_\alpha(\lambda)\leq 1$, $\alpha>0$,
$0\leq \lambda\leq \lambda_1$;

\quad \textbf{b)} $r_\alpha(\lambda)\geq\gamma_1$,
$0\leq\lambda<\alpha\leq \lambda_1$;

\quad \textbf{c)} $\abs{r_\alpha(\lambda)}$ is monotone increasing
with respect to $\alpha$ for $\lambda\in (0,\norm{T}^2]$;

\quad \textbf{d)} $g_\alpha(c_1\alpha)\geq
\frac{\gamma_2}{\alpha}$, $0<c_1\alpha\leq \lambda_1$ and

\quad \textbf{e)} $g_\alpha(\lambda)\geq
g_\alpha(\tilde{\lambda})$, for $0<\alpha\leq\lambda\leq
\tilde{\lambda}\leq \lambda_1$.

There exist constants $\gamma,c>0$ such that:

\textbf{iii)} The family of regularization operators
$\{R_\alpha\}_{\alpha\in (0,\alpha_0)}$ defined by
(\ref{eq:Ralfa-galfa}), where $\alpha_0\doteq\min\{\lambda_1,
\frac{\lambda_1}{c}\}$, has classical qualification of order
$\mu_0, \;0<\mu_0<+\infty$.

\textbf{iv)}
\begin{equation}\label{cond-450}
    \left(\frac{\lambda}{\alpha}\right)^{\mu_0}\abs{r_\alpha(\lambda)}\geq
    \gamma, \quad \textrm{for every}\; 0<c\alpha\leq \lambda \leq
    \norm{T}^2.
\end{equation}

Then $\psi_{\mu_0}(x,\delta)\doteq
\delta^{\frac{2\mu_0}{2\mu_0+1}}$ for $x \in X_{\mu_0}\doteq
\mathcal{R}((T^\ast T)^{\mu_0})\setminus \{0\}$ and $\delta
>0$, is saturation function of
$\{R_\alpha\}_{\alpha\in (0,\alpha_0)}$ on $X_{\mu_0}$.
\end{thm}

\smallskip \smallskip

Note that the hypothesis \textbf{\textit{i)}} is trivially
satisfied if $T$ is compact. To prove this theorem we will need
two previous lemmas. In the first one we show that under the
hypotheses of Theorem \ref{teo:sat-espectral}, for all $\alpha$ in
a right neighborhood of zero one has that
$0\in\rho\left(r_\alpha(T^\ast T)\right)$, i.e. zero belongs to
the resolvent set of the operator $r_\alpha(T^\ast T)$. More
precisely we have the following:

\smallskip
\begin{lem}\label{lem:r-invertible}
Suppose that $\{g_\alpha\}_{\alpha>0}$ satisfies hypotheses
\textit{H1-H4} and assume further that hypotheses
\textbf{\textit{ii.b)}}, \textbf{\textit{ii.c)}},
\textbf{\textit{iii)}} and \textbf{\textit{vi)}} of Theorem
\ref{teo:sat-espectral} hold. Then for every $\alpha \in
(0,\alpha_0)$ the operator $r_\alpha(T^\ast T)$ is invertible,
where $\alpha_0\doteq\min\{\lambda_1, \frac{\lambda_1}{c}\}$.
\end{lem}

\smallskip
\begin{proof}
It suffices to show that for every $\alpha \in (0,\alpha_0)$ and
for every $x \in X$, the function $r_\alpha^{-2}(\lambda)$ is
integrable with respect to the measure $d\norm{E_\lambda x}^2$.
Let $\alpha \in (0,\alpha_0)$ be arbitrary but fixed. Since
$\alpha_0\leq \lambda_1$, it follows from hypothesis
\textbf{\textit{ii.b)}} that $r_\alpha(\lambda)\geq \gamma_1>0$
for every  $\lambda \in [0, \alpha)$. Then
\begin{equation}\label{eq:int1}
\int_0^\alpha \frac 1 {r_\alpha^2(\lambda)}\;d\norm{E_\lambda
    x}^2\leq  \frac{\norm{x}^2}{\gamma_1^2}<+\infty.
\end{equation}

It remains to prove that $\int_\alpha^{\norm{T}^2+} \frac 1
{r_\alpha^2(\lambda)}\;d\norm{E_\lambda
    x}^2<+\infty.$ For that we shall consider two cases.

\underline{Case I}: $c\leq 1.$ In this case, for every $\lambda
\in [\alpha,\norm{T}^2]$ we have that $\lambda \geq \alpha \geq
c\,\alpha
>0$ and from (\ref{cond-450}) it follows that $\abs{r_\alpha(\lambda)}\geq
\gamma \left(\frac{\alpha}{\lambda}\right)^{\mu_0}$ for every
$\lambda \in [\alpha, \norm{T}^2]$. Therefore
\begin{equation*}
\int_{\alpha}^{\norm{T}^2+} \frac 1
{r_\alpha^2(\lambda)}\;d\norm{E_\lambda
    x}^2\leq \int_{\alpha}^{\norm{T}^2+}
\frac{\lambda^{2\mu_0}}{(\alpha^{\mu_0}\gamma)^{2}}\;
d\norm{E_\lambda x}^2\leq \frac{\norm{(T^\ast
    T)^{\mu_0}x}^2}{(\alpha^{\mu_0}\gamma)^{2}}<+\infty.
\end{equation*}

\underline{Case II}: $c>1$. In this case, since $c\alpha<\|T\|^2$
we write
\begin{equation}\label{eq:int2}
\int_{\alpha}^{\norm{T}^2+} \frac 1
{r_\alpha^2(\lambda)}\;d\norm{E_\lambda x}^2=
\int_{\alpha}^{c\,\alpha} \frac 1
{r_\alpha^2(\lambda)}\;d\norm{E_\lambda x}^2 +
\int_{c\,\alpha}^{\norm{T}^2+} \frac 1
{r_\alpha^2(\lambda)}\;d\norm{E_\lambda
    x}^2.
\end{equation}

Like in the previous case, by virtue of (\ref{cond-450}), the
second integral on the RHS of (\ref{eq:int2}) is bounded above by
$\frac{\norm{(T^\ast
T)^{\mu_0}x}^2}{(\alpha^{\mu_0}\gamma)^{2}}<+\infty.$ For the
first integral on the RHS of (\ref{eq:int2}), by virtue of
hypothesis \textbf{\textit{ii.c)}} we have that
\begin{equation}\label{eq:1}
    r_\alpha^2(\lambda)\geq r_{\alpha/c}^2(\lambda), \quad
    \forall\; \lambda\in[\alpha, c\alpha]
\end{equation}
because $\frac{\alpha}{c}<\alpha$. On the other hand, again by
using (\ref{cond-450}), and given that $0<c(\frac{\alpha}{c})\leq
\lambda$ we have that
\begin{equation} \label{eq:2}
\left(\frac{\lambda}{\alpha/c}\right)^{2\mu_0}r_{\alpha/c}^2(\lambda)\geq
    \gamma^2.
\end{equation}
From (\ref{eq:1}) and (\ref{eq:2}) we conclude that
$r_\alpha^2(\lambda)\geq \gamma^2
\left(\frac{\alpha}{c\,\lambda}\right)^{2\mu_0}$ for every
$\lambda \in [\alpha,c\,\alpha]$. Thus, for the first integral on
the RHS of (\ref{eq:int2}) we have the estimate
\begin{equation*}
\int_{\alpha}^{c\,\alpha}\frac 1
{r_\alpha^2(\lambda)}\;d\norm{E_\lambda
    x}^2\leq \int_{\alpha}^{c\,\alpha}
\frac{c^{2\mu_0}}{\alpha^{2\mu_0}\gamma^2}\lambda^{2\mu_0}\;
d\norm{E_\lambda
    x}^2\leq \frac{c^{2\mu_0}}{\alpha^{2\mu_0}\gamma^2}\norm{(T^\ast
    T)^{\mu_0}x}^2<\infty.
\end{equation*}
Hence $r_\alpha(T^\ast T)$ is an invertible operator for every
$\alpha\in(0, \alpha_0)$.\hfill
\end{proof}

\smallskip
\begin{lem}\label{lem:o-O}
Suppose that $\{g_\alpha\}_{\alpha>0}$ satisfies the hypotheses
\textit{H1-H4} and suppose further that hypotheses
\textbf{\textit{ii.b)}}, \textbf{\textit{ii.c)}},
\textbf{\textit{iii)}} and \textbf{\textit{iv)}} of Theorem
\ref{teo:sat-espectral} hold. Let
$\varphi:[0,\norm{T}^2]\rightarrow\R^+$ be a continuous, strictly
increasing function satisfying $\varphi(0)=0$. If for some $x^\ast
\in X$, $x^\ast \neq 0$ we have that
$\et(x^\ast,\delta)=o(\varphi(\delta))$ for $\delta \rightarrow
0^+$, then there exists an \textit{a-priori} parameter choice rule
$\tilde{\alpha}(\delta)$ such that
\begin{equation*}
\underset{y^\delta \in \overline{B_\delta(Tx^\ast)}}{\sup}
\norm{R_{\tilde{\alpha}(\delta)}
y^\delta-x^\ast}=o(\varphi(\delta))\quad \textrm{ for } \delta
\rightarrow 0^+.
\end{equation*}
The same remains true if we replace $o(\varphi(\delta))$ by
$O(\varphi(\delta))$.
\end{lem}

\smallskip
\begin{proof}
Let $\varphi$ be as in the hypotheses and suppose that there
exists $x^\ast \in X$, $x^\ast\neq 0$ such that
$\et(x^\ast,\delta)=o(\varphi(\delta))$ for $\delta \rightarrow
0^+$. Then by definition of $\et,$
\begin{equation} \label{eq:11}
\underset{\delta \rightarrow 0^+}{\lim}\frac{\underset{\alpha \in
(0,\alpha_0)}{\inf}\;\underset{y^\delta \in
\overline{B_\delta(Tx^\ast)}}{\sup} \norm{R_\alpha
y^\delta-x^\ast}}{\varphi(\delta)}=\underset{\delta \rightarrow
0^+}{\lim}\,\underset{\alpha \in
(0,\alpha_0)}{\inf}\frac{\underset{y^\delta \in
\overline{B_\delta(Tx^\ast)}}{\sup} \norm{R_\alpha
y^\delta-x^\ast}}{\varphi(\delta)}=0.
\end{equation}
For the sake of simplify we introduce the following notation:
\begin{equation*}
f(\alpha,\delta)\doteq\frac{\underset{y^\delta \in
\overline{B_\delta(Tx^\ast)}}{\sup} \norm{R_\alpha
y^\delta-x^\ast}}{\varphi(\delta)}\quad \textrm{and}\quad
h(\delta)\doteq \underset{\alpha \in (0,\alpha_0)}{\inf}
f(\alpha,\delta).
\end{equation*}
Then $h(\delta)>0$ for every $\delta \in(0,\infty)$ and
(\ref{eq:11}) can be written simply as $\underset{\delta
\rightarrow 0^+}{\lim}h(\delta)=0$. Next, for $n \in \N$ we define
\begin{equation*}
\delta_n\doteq \sup \left\{\delta>0: h(\delta)\leq \frac 1 n
\right\}.
\end{equation*}
Clearly, $\delta_n \downarrow 0$ and $h(\delta)=\underset{\alpha
\in (0,\alpha_0)}{\inf} f(\alpha,\delta)\leq \frac 1 n$ for every
$\delta \in (0, \delta_n]$ for every $n \in \N$. Then, there
exists $\alpha_n=\alpha_n(\delta_n)\in (0,\alpha_0)$ such that
\begin{equation}\label{eq:desig}
f(\alpha_n,\delta)\leq \frac 2 n \quad\forall\; \delta \in (0,
\delta_n], \;\forall \;n \in \N.
\end{equation}
We then define $\alpha(\delta)\doteq \alpha_n$ for all $\delta \in
(\delta_{n+1},\delta_n]$ for every $n \in \N$. Then, since
$\delta_n \downarrow 0$ it follows from (\ref{eq:desig}) that
$\underset{\delta \rightarrow
0^+}{\lim}f(\alpha(\delta),\delta)=\underset{n\rightarrow
+\infty}{\lim}f(\alpha_n,\delta_n)=0$. We could choose $\alpha$ as
the parameter choice rule we are looking for. The problem is that
we cannot guarantee the existence of the limit of $\alpha(\delta)$
for $\delta \rightarrow 0^+$. However, we will see next that
$\alpha(\delta)$ can be replaced by a function
$\tilde{\alpha}:\R^+\rightarrow(0,\alpha_0)$ such that
$\underset{\delta\rightarrow 0^+}{\lim}\,\tilde{\alpha}(\delta)=0$
(i.e, such that $\tilde{\alpha}(\delta)$ is an admissible
parameter choice rule) maintaining the condition $\underset{\delta
\rightarrow 0^+}{\lim}f(\tilde{\alpha}(\delta),\delta)=0$. In
fact, since $\{\alpha_n\}_{n \in \N}\subset (0,\alpha_0)$ is a
bounded sequence of real numbers, it contains a convergent
subsequence $\{\alpha_{n_k}\}_{k \in \N}$, with
$\alpha_{n_k}\rightarrow \alpha^\ast$ for $k\rightarrow +\infty$,
and some $\alpha^\ast \in [0,\alpha_0].$ We define
$\tilde{\alpha}(\delta)\doteq \alpha_{n_k}$ for all $\delta \in
(\delta_{n_{k+1}},\delta_{n_k}]$, for every $k \in \N$. Then,
\begin{equation}\label{eq:alpha-ast}
\underset{\delta\rightarrow
0^+}{\lim}\,\tilde{\alpha}(\delta)=\underset{k\rightarrow
+\infty}{\lim}\,\alpha_{n_k}=\alpha^\ast.
\end{equation}
Since $\{\alpha_{n_k}\}_{k \in \N}$ and $\{\delta_{n_k}\}_{k \in
\N}$ are subsequences of $\{\alpha_{n}\}_{n \in \N}$ and
$\{\delta_{n}\}_{n \in \N}$,\break $\underset{\delta\rightarrow
0^+}{\lim}f(\tilde{\alpha}(\delta),\delta)=\underset{k\rightarrow
+\infty}{\lim}f(\alpha_{n_k},\delta_{n_k})=0.$ Then, by definition
of $f$,
\begin{equation*}
\underset{\delta \rightarrow 0^+}{\lim}\,\frac{\underset{y^\delta
\in \overline{B_\delta(Tx^\ast)}}{\sup}
\norm{R_{\tilde{\alpha}(\delta)}
y^\delta-x^\ast}}{\varphi(\delta)}=0,
\end{equation*}
that is,
\begin{equation}\label{eq:lim}
\underset{y^\delta \in \overline{B_\delta(Tx^\ast)}}{\sup}
\norm{R_{\tilde{\alpha}(\delta)}
y^\delta-x^\ast}=o(\varphi(\delta)), \;\textrm{ as }\; \delta
\rightarrow 0^+.
\end{equation}
It remains to be shown that $\alpha^\ast=0$. If $\alpha^\ast
>0$, then it follows from (\ref{eq:alpha-ast}) that there exists
$\delta_0>0$ such that
$\tilde{\alpha}(\delta)>\frac{\alpha^\ast}{2}$ for all $\delta \in
(0,\delta_0)$. Hypothesis \textbf{\textit{ii.c)}} of Theorem
\ref{teo:sat-espectral} implies then that for every $\delta \in
(0,\delta_0)$, $\abs{r_{\tilde{\alpha}(\delta)}(\lambda)}\geq
\abs{r_{\frac{\alpha^\ast}{2}}(\lambda)}$ for all $\lambda \in (0,
\norm{T}^2].$ It follows that for every $\delta \in (0,\delta_0)$,
\begin{eqnarray*}
  \norm{r_{\tilde{\alpha}(\delta)}(T^\ast
T)x^\ast}^2 &=&
\int_0^{\norm{T}^2+}r_{\tilde{\alpha}(\delta)}^2(\lambda)\;d\norm{E_\lambda
    x^\ast}^2 \\ \nonumber
&\geq&\int_0^{\norm{T}^2+}r_{\frac{\alpha^\ast}{2}}^2(\lambda)\;d\norm{E_\lambda
    x^\ast}^2 \\
    &=&\norm{r_{\frac{\alpha^\ast}{2}}(T^\ast
T)x^\ast}^2.
\end{eqnarray*}
Then, for all $\delta \in (0,\delta_0)$,
\begin{eqnarray*}
\underset{y^\delta \in \overline{B_\delta(Tx^\ast)}}{\sup}
\norm{R_{\tilde{\alpha}(\delta)} y^\delta-x^\ast} &\geq&
\norm{R_{\tilde{\alpha}(\delta)} Tx^\ast-x^\ast}
   = \norm{(I-g_{\tilde{\alpha}(\delta)}(T^\ast T)T^\ast T)x^\ast} \\
   &=& \norm{r_{\tilde{\alpha}(\delta)}(T^\ast T)x^\ast}
   \geq\norm{r_{\frac{\alpha^\ast}{2}}(T^\ast T)x^\ast}.
\end{eqnarray*}
Taking limit for $\delta \rightarrow 0^+$ and using (\ref{eq:lim})
we conclude that $\norm{r_{\frac{\alpha^\ast}{2}}(T^\ast
T)x^\ast}=0.$ But since $\frac{\alpha^\ast}{2}<\alpha_0$, it
follows from Lemma \ref{lem:r-invertible} that
$r_{\frac{\alpha^\ast}{2}}(T^\ast T)$ is invertible and therefore
$x^\ast=0$, which is a contradiction since $x^\ast$ was not zero
to start with. Hence, $\alpha^\ast$ must be equal to zero, as
wanted.

We proceed now to prove the second part of the Lemma. Suppose that
there exists $x^\ast \in X$, $x^\ast\neq 0$ such that
$\et(x^\ast,\delta)=O(\varphi(\delta))$ as $\delta \rightarrow
0^+$. Then there exist positive constants $k$ and $d$ such that
$\underset{\alpha \in (0,\alpha_0)}{\inf}f(\alpha, \delta)\leq k$
for every $\delta \in (0,d)$, where $f(\alpha,\delta)$ is as
previously defined. Let $\{\delta_n\}_{n\in \N}\subset (0,d)$ be
such that $\delta_n\downarrow 0$ and
$\alpha_n=\alpha_n(\delta_n)\in (0,\alpha_0)$ such that
\begin{equation*}
    f(\alpha_n,\delta)\leq k+\delta_n, \;\forall\, \delta \in
    (0,d),\;\forall \,n \in \N.
\end{equation*}
We define (just like we did it previously for the ``$o$'' case)
$\alpha(\delta)\doteq \alpha_n$ for all $\delta \in
(\delta_{n+1},\delta_n]$ for every $n \in \N$. Since $\delta_n
\downarrow 0$ it follows that $f(\alpha(\delta),\delta)\leq
k+\delta_1$ for every $\delta \in (0,d)$ and therefore
\begin{equation}\label{eq:O}
\underset{y^\delta \in \overline{B_\delta(Tx^\ast)}}{\sup}
\norm{R_{\alpha(\delta)} y^\delta-x^\ast}=O(\varphi(\delta))\quad
\textrm{as } \delta \rightarrow 0^+.
\end{equation}

Exactly in the same way as we proceeded before in the first part
of the proof, by defining the function $\tilde{\alpha}(\delta)$
(from a convergent subsequence of $\{\alpha_n\}_{n \in \N}$),
equation (\ref{eq:O}) is proved with $\tilde{\alpha}(\delta)$ in
place of $\alpha(\delta)$. Finally, and also by proceeding in an
analogous way, it is shown that $\tilde{\alpha}(\delta)$ converges
to zero as $\delta \rightarrow 0^+$, i.e. that
$\tilde{\alpha}(\delta)$ is an admissible parameter choice rule.
Since the steps are essentially the same we do not give details
here.\hfill
\end{proof}

We are now ready to prove Theorem \ref{teo:sat-espectral}.

\textit{Proof of Theorem \ref{teo:sat-espectral}.} We will show
that $\psi_{\mu_0}(x,\delta)\doteq
\delta^{\frac{2\mu_0}{2\mu_0+1}}$ for $x \in X_{\mu_0}$ and
$\delta >0$, is saturation function of $\{R_\alpha\}_{\alpha\in
(0,\alpha_0)}$ on $X_{\mu_0}$.

First we note that by virtue of Lemma \ref{lem:cotasup-xmu},
$\psi_{\mu_0} \in \mathcal{U}_{X_{\mu_0}}(\et)$. Next we will show
that $\psi_{\mu_0}$ satisfies condition \textit{S1} of saturation
on $X_{\mu_0}$ (see Definition \ref{def:satur}). Suppose that it
is not true, i.e. suppose that there exist $x^\ast \in X$,
$x^\ast\neq 0$ and $x \in X_{\mu_0}$ such that $\underset{\delta
\rightarrow0^+}{\limsup}\,\frac{\et(x^\ast,\delta)}{\psi_{\mu_0}(x,\delta)}=0$.
Then
$\et(x^\ast,\delta)=o\left(\delta^{\frac{2\mu_0}{2\mu_0+1}}\right)$
as $\delta \rightarrow 0^+$ and from Lemma \ref{lem:o-O} it
follows that there exists an \textit{a-priori} admissible
parameter choice rule
 $\alpha(\delta)$ such that
\begin{equation*}
\underset{y^\delta \in \overline{B_\delta(Tx^\ast)}}{\sup}
\norm{R_{\alpha(\delta)}
y^\delta-x^\ast}=o(\delta^{\frac{2\mu_0}{2\mu_0+1}})\quad \textrm{
for } \delta \rightarrow 0^+.
\end{equation*}

Now note that hypothesis  \textit{H4} implies that there exists a
finite positive constant $\beta$ such that
$\sqrt{\lambda}\abs{g_\alpha(\lambda)}\leq
\frac{\beta}{\sqrt{\alpha}}$, for every $\alpha \in (0,\alpha_0)$
and for every $\lambda \in [0,\norm{T}^2]$. Since $\{g_\alpha\}$
satisfies the hypotheses \textit{H1-H4} and
\textbf{\textit{i)-iv)}} hold, it follows from Theorem 3.1 of
\cite{Neubauer94} that $x^\ast=0$, which contradicts the fact that
$x^\ast$ was different from zero. Hence, $\psi_{\mu_0}$ satisfies
condition \textit{S1} on $X_{\mu_0}$. Since $\psi_{\mu_0}$ does
not depend on $x$, we further have that $\psi_{\mu_0}$ is
(trivially) invariant over $X_{\mu_0}$, i.e., it satisfies
condition \textit{S2}.

It only remains to prove that $\psi_{\mu_0}$ satisfies condition
\textit{S3}, that is, that the set $X_{\mu_0}$ is optimal for
$\psi_{\mu_0}$. Suppose that is not the case. Then there must
exist $M\supsetneqq X_{\mu_0}$ and $\tilde{\psi} \in
\mathcal{U}_{M}(\et)$ such that $\tilde{\psi}\mid
_{X_{\mu_0}}=\psi_{\mu_0}$ and $\tilde{\psi}$ satisfies
\textit{S1} and \textit{S2} on $M$. Let $x^\ast \in M\setminus
X_{\mu_0}$, $x^\ast \neq 0$. Since $\tilde{\psi} \in
\mathcal{U}_{M}(\et)$ we have that
\begin{equation}\label{eq:des} \et
\overset{\{x^\ast\}}{\preceq}\tilde{\psi}.
\end{equation}
Also, since $\tilde{\psi}$ is invariant over $M$, we have that
$\tilde{\psi}\overset{\{x^\ast\},X_{\mu_0}}{\preceq}{\tilde{\psi}}$,
and since $\tilde{\psi}$ coincides with $\psi_{\mu_0}$ on
$X_{\mu_0}$, it follows that
$\tilde{\psi}\overset{\{x^\ast\},X_{\mu_0}}{\preceq}{\psi_{\mu_0}}$.
This, together with (\ref{eq:des}) implies that
$\et\overset{\{x^\ast\},X_{\mu_0}}{\preceq}{\psi_{\mu_0}}$ and
therefore
$\et(x^\ast,\delta)=O\left(\delta^{\frac{2\mu_0}{2\mu_0+1}}\right)$
as $\delta \rightarrow 0^+.$ Lemma \ref{lem:o-O} then implies that
there exists an \textit{a-priori} admissible parameter choice rule
 $\alpha(\delta)$ such that
\begin{equation*}
    \underset{y^\delta \in \overline{B_\delta(Tx^\ast)}}{\sup}
\norm{R_{\alpha(\delta)}
y^\delta-x^\ast}=O\left(\delta^{\frac{2\mu_0}{2\mu_0+1}}\right)
\quad \textrm{as } \delta \rightarrow 0^+.
\end{equation*}
Since $\mu_0<+\infty$ it follows that $x^\ast \in
\mathcal{R}((T^\ast T)^{\mu_0})$ (see \cite{Neubauer94}, Corollary
2.6) and since $x^\ast \neq 0$, we have that $x^\ast \in
X_{\mu_0}$ which contradicts that $x^\ast \in M \setminus
X_{\mu_0}$. Thus, $\psi_{\mu_0}$ satisfies condition \textit{S3}
and $\psi_{\mu_0}$ is saturation function of $\{R_\alpha\}$ on
$X_{\mu_0}$, as we wanted to prove. \hfill

\subsection{Spectral Methods with Maximal Qualification}
The concept of classical qualification is a special case of a more
general definition of qualification introduced by Math\'{e} and
Pereverzev (\cite{ref:Mathe-Pereverzev-2003}, \cite{Mathe2004}).

\begin{defn}\label{def:calif-mathe}
Let $\{R_\alpha\}_{\alpha \in (0,\alpha_0)}$ be a family of
spectral regularization operators for $Tx=y$ generated by the
family of functions $\{g_\alpha\}_{\alpha \in (0,\alpha_0)}$ and
let $r_\alpha(\lambda)\doteq 1-\lambda g_\alpha(\lambda)$. A
function  $\rho:(0,\norm{T}^2]\rightarrow \R^+$ is said to be
qualification of $\{R_\alpha\}_{\alpha \in (0,\alpha_0)}$ if
$\rho$ is increasing and there exists a constant $\gamma>0$ such
that
\begin{equation*}
\underset{\lambda \in
(0,\norm{T}^2]}{\sup}\abs{r_\alpha(\lambda)}\,\rho(\lambda)\leq
\gamma \,\rho(\alpha) \quad \textrm{for every}\; \alpha \in
(0,\alpha_0).
\end{equation*}
If, moreover, for every $\lambda \in (0,\norm{T}^2]$ there exists
a constant $c\doteq c(\lambda)>0$ such that
\begin{equation*}
\underset{\alpha\in
(0,\alpha_0)}{\inf}\frac{\abs{r_\alpha(\lambda)}}{\rho(\alpha)}\geq
c
\end{equation*}
then $\rho$ is said to be maximal qualification of
$\{R_\alpha\}_{\alpha \in (0,\alpha_0)}$.
\end{defn}

Note then that the classical qualification of order $\mu$
corresponds to the case in which the functions $\rho$ are
restricted to monomials $\rho(t)=t^\mu$ for $0\leq \mu < +\infty$.

These two definitions of qualification are closely related. For
instance, if a spectral regularization method $\{R_\alpha\}$
possesses classical qualification of order $\mu_0<\infty$, then
any increasing function $\tilde{\rho}:(0,\norm{T}^2]\rightarrow
\R^+$ satisfying $\alpha^{\mu_0}\leq k\,\tilde{\rho}(\alpha)$ for
some constant $k>0$, for $\alpha$ in a neighborhood of $\alpha=0$,
is also qualification of $\{R_\alpha\}$. Also, if $\alpha^{\mu_0}$
and $\tilde{\rho}(\alpha)$ are two maximal qualifications then
they are necessarily equivalent in the sense that there exist
constants $k, \tilde{k}>0$ such that $k\,\alpha^{\mu_0}\leq
\tilde{\rho}(\alpha)\leq \tilde{k}\, \alpha^{\mu_0}$ for every
$\alpha\in(0,\alpha_0)$. On the other hand, if a spectral
regularization method $\{R_\alpha\}$ has classical qualification
of infinite order, then it does not necessarily have maximal
qualification.

Next, we will show that under certain general hypotheses, it is
also possible to characterize the saturation of spectral
regularization methods possessing maximal qualification.  For that
we will previously need the following definition.

\begin{defn}\label{def:local-type}
Let $\rho:(0,a]\rightarrow(0,+\infty)$ be a continuous
non-decreasing function such that $\underset{t\rightarrow
0^+}{\lim}\rho(t)=0$ and $\beta \in \R$, $\beta \geq 0$. We say
that $\rho$ is of local upper type $\beta$ if there exists a
positive constant $d$ such that $\rho(t)\leq d\,(\frac 1
s)^\beta\rho(s\,t)$ for every $s \in (0,1]$, $t \in (0,a]$.
\end{defn}

A function of finite upper type is also said to satisfy a
$\Delta_\beta$ condition.

\begin{thm}\label{teo:sat-calmaxima} (Saturation for families of spectral
regularization operators with maximal qualification.)

Let $T$ be a compact linear operator. Suppose that
$\{g_\alpha\}_{\alpha \in (0,\alpha_0)}$ satisfies hypotheses
H1-H4 and let $\{R_\alpha\}_{\alpha\in (0,\alpha_0)}$ be as
defined by (\ref{eq:Ralfa-galfa}). Suppose further that the
following hypotheses are satisfied:

\textbf{M1}: There exist
$\{\tilde{\lambda}_n\}_{n=1}^\infty\subset \sigma_p(TT^\ast)$ and
$c\geq 1$ such that $\tilde{\lambda}_n \downarrow 0$ and
$\frac{\tilde{\lambda}_n}{\tilde{\lambda}_{n+1}}\leq c$ for every
$n\in \N$.

\textbf{M2}: There exist positive constants $\lambda_1 \leq
\norm{T}^2$, $\gamma_1,\gamma_2$ and $c_1>1$ such that

\quad \textbf{a)} $0 \leq r_\alpha(\lambda)\leq 1$, $\alpha>0$,
$0\leq \lambda\leq \lambda_1$;

\quad \textbf{b)} $r_\alpha(\lambda)\geq\gamma_1$,
$0\leq\lambda<\alpha\leq \lambda_1$;

\quad \textbf{c)} $\abs{r_\alpha(\lambda)}$ is monotone increasing
as a function of $\alpha$ for each $\lambda\in (0,\norm{T}^2]$;

\quad \textbf{d)} $g_\alpha(c_1\alpha)\geq
\frac{\gamma_2}{\alpha}$, $0<c_1\alpha\leq \lambda_1$ and

\quad \textbf{e)} $g_\alpha(\lambda)\geq
g_\alpha(\tilde{\lambda})$, for $0<\alpha\leq\lambda\leq
\tilde{\lambda}\leq \lambda_1$.

\textbf{M3}: There exists $\rho:(0,\norm{T}^2]\rightarrow
(0,+\infty)$, strictly increasing and of local upper type $\beta$,
for some $\beta \geq0$, such that $\rho$ is maximal qualification
of $\{R_\alpha\}_{\alpha\in (0,\alpha_0)}$ and there exist
positive constants $a$ and $k$ such that
\begin{equation*}
\frac{\rho(\lambda)\abs{r_\alpha(\lambda)}}{\rho(\alpha)}\geq
   a, \quad \textrm{for all}\; \alpha,\,\lambda \; \textrm{such that } \; 0<k\,\alpha\leq \lambda \leq
    \norm{T}^2.
\end{equation*}

\vspace{-.13in} \textbf{M4}: For every $\alpha \in (0,\alpha_0)$
the function $\lambda \rightarrow \abs{r_\alpha(\lambda)}^2$,
$\lambda \in (0,\norm{T}^2]$ is convex.

\smallskip

Let $\Theta(t)\doteq \sqrt{t}\rho(t)$ for $t \in (0,\norm{T}^2]$.
Then $\psi(x,\delta)\doteq (\rho \circ \Theta^{-1})(\delta)$ for
$x \in X^\rho\doteq \mathcal{R}(\rho(T^\ast T))\setminus \{0\}$
and $\delta \in (0,\Theta (\alpha_0))$, is saturation function of
$\{R_\alpha\}_{\alpha \in (0,\alpha_0)}$ on $X^\rho$.
\end{thm}

\smallskip\smallskip
In order to prove this theorem we will previously need two
converse results that we establish in the following two Lemmas.

\smallskip
\begin{lem}\label{lema:res-reciproco}
Let $\{R_\alpha\}_{\alpha \in (0,\alpha_0)}$ be a family of
spectral regularization operators for $Tx=y$ and
$\rho:(0,\norm{T}^2]\rightarrow \R^+$ a strictly increasing
continuous function satisfying hypothesis \textbf{\textit{M3}} of
Theorem \ref{teo:sat-calmaxima}. If for some $x \in X$,
$\norm{R_\alpha Tx-x}=O(\rho(\alpha))$ for $\alpha\rightarrow
0^+$, then $x \in \mathcal{R}(\rho(T^\ast T)).$
\end{lem}

\smallskip
\begin{proof}
From hypothesis \textbf{\textit{M3}} it follows that
\begin{equation}\label{eq:13}
\norm{R_\alpha Tx-x}^2 = \int_0^{\norm{T}^2+}r_\alpha^2(\lambda)\,
d\norm{E_\lambda x}^2\geq a^2\,\rho^2(\alpha)\int_{k
\,\alpha}^{\norm{T}^2+}\rho^{-2}(\lambda)\,d\norm{E_\lambda x}^2.
\end{equation}
Since $\norm{R_{\alpha} Tx-x}=O(\rho(\alpha))$ for
$\alpha\rightarrow 0^+$, it then follows that there are constants
$C>0$ and $\alpha^\ast$, $0<\alpha^\ast \leq \alpha_0$ such that
\begin{equation*}
\int_{k\,\alpha}^{\norm{T}^2+}\rho^{-2}(\lambda)\,d\norm{E_\lambda
x}^2 \leq \frac{\norm{R_\alpha Tx-x}^2}{a^2\rho^2(\alpha)}\leq
\frac{C^2}{a^2} \quad \textrm{for every}\; \alpha \in (0,
\alpha^\ast).
\end{equation*}
Taking limit for $\alpha \rightarrow 0^+$ we obtain that
$\int_{0}^{\norm{T}^2+}\rho^{-2}(\lambda)\,d\norm{E_\lambda x}^2<
+ \infty$, from which it follows that $w\doteq
\int_0^{\norm{T}^2+}\rho^{-1}(\lambda)\,dE_{\lambda}\,x \in X.$
Then,
\begin{equation*}
\rho(T^\ast
T)w=\int_0^{\norm{T}^2+}\rho(\lambda)\rho^{-1}(\lambda)\,dE_{\lambda}x=x
\end{equation*}
and therefore $x \in \mathcal{R}(\rho(T^\ast T))$.\hfill
\end{proof}

\smallskip
\begin{lem}\label{lem:sup-inf} Under the same hypotheses of Theorem
\ref{teo:sat-calmaxima}, if for some $x \in X$ we have that
\begin{equation}\label{eq:19} \underset{y^\delta \in
\overline{B_\delta(Tx)}}{\sup}\;\underset{\alpha \in
(0,\alpha_0)}{\inf} \norm{R_{\alpha}
y^\delta-x}=O(\rho(\Theta^{-1}(\delta))) \quad \textrm{when }
\delta \rightarrow 0^+,
\end{equation}
then $x \in \mathcal{R}(\rho(T^\ast T))$.
\end{lem}

\smallskip
\begin{proof}
Without loss of generality we assume that
$\alpha_0\le\min\{\frac{\lambda_1}{c_1}, \frac{\lambda_1}{k}\}$
and that $x \neq 0$ (if $x=0$ the result is trivial).

Let $\bar{\lambda}\in \sigma_p(TT^\ast)$ be such that
$0<c_1\,\bar{\lambda}\leq \lambda_1$ (the compactness of $T$
guarantees the existence of such $\bar{\lambda}$), and define
\begin{equation*}
    \bar{\delta}=\bar{\delta}(\bar\lambda)\doteq
    \frac{\bar{\lambda}^{1/2}}{\gamma_2}\norm{R_{\bar{\lambda}}Tx-x}.
\end{equation*}
Then, clearly the equation
\begin{equation}\label{eq:14}
\norm{R_\alpha Tx -x}^2=\frac{(\gamma_2\,\bar{\delta})^2}{\alpha}
\end{equation}
in the unknown $\alpha$, has $\alpha=\bar{\lambda}$ as a solution.
Moreover, from the hypothesis \textbf{\textit{M2\;c)}} and given
that $x \neq 0$, it follows that $\alpha =\bar{\lambda}$ is the
unique solution of (\ref{eq:14}). Note also that
$\bar{\delta}\rightarrow 0^+$ if (and only if)
$\bar{\lambda}\rightarrow 0^+$.

Now, for $\delta>0$ define
\begin{equation}\label{eq:16}
    \bar{y}^{\,\delta}\doteq Tx-\delta G_{\bar{\lambda}}z, \quad
    \forall\; \delta>0,
\end{equation}
where $G_{\bar{\lambda}}\doteq
F_{c_1\,\bar{\lambda}}-F_{\bar{\lambda}}$ and $\{F_\lambda\}$ is
the spectral family associated to $TT^\ast$ and
\begin{equation*}
    z\doteq \left\{%
\begin{array}{ll}
    \norm{G_{\bar{\lambda}}Tx}^{-1}Tx, & \hbox{if $G_{\bar{\lambda}}Tx\neq
0$,} \\
    \textrm{arbitrary with} \norm{G_{\bar{\lambda}}z}=1, &
    \hbox{in other case.} \\
\end{array}%
\right.
\end{equation*}
Note that since $\bar{\lambda}\in \sigma_p(TT^\ast)$ and $c_1>1$
it follows that $G_{\bar{\lambda}}$ is not the null operator and
therefore the definition makes sense. Note also that
$\norm{\bar{y}^{\,\delta}-Tx}=\delta$, which implies that
$\bar{y}^{\,\delta} \in \overline{B_\delta(Tx)}$.

Now, from (\ref{eq:Ralfa-galfa}), (\ref{eq:16}) and from the fact
that $g_\alpha(T^\ast T)T^\ast= T^\ast g_\alpha(TT^\ast)$ it
follows that for every $\alpha \in (0,\alpha_0)$ and $\delta>0$,
\begin{eqnarray}\label{eq:12} \nonumber
\seq{R_\alpha Tx-x\right.\left.,R_\alpha (\bar{y}^{\,\delta}-Tx)}
&=& \seq{g_\alpha(T^\ast T)T^\ast Tx-x
,-g_\alpha(T^\ast T)T^\ast\,\delta G_{\bar{\lambda}}z} \\
\nonumber  &=& \delta \seq{g_\alpha(T^\ast T)T^\ast Tx-x,-T^\ast
g_\alpha(TT^\ast)G_{\bar{\lambda}}z} \\ \nonumber &=& \delta
\seq{Tg_\alpha(T^\ast T)T^\ast
Tx-Tx,-g_\alpha(TT^\ast)G_{\bar{\lambda}}z} \\ \nonumber  &=&
\delta \seq{(TT^\ast
g_\alpha(TT^\ast)-I)Tx,-g_\alpha(TT^\ast)G_{\bar{\lambda}}z}
\\ \nonumber
&=& \delta
\seq{-r_\alpha(TT^\ast)Tx,-g_\alpha(TT^\ast)G_{\bar{\lambda}}z}\\
&=& \delta \int_0^{\norm{T}^2+}
   r_\alpha(\lambda)g_\alpha(\lambda)\,d\seq{F_\lambda
   Tx,G_{\bar{\lambda}}z}.
\end{eqnarray}

Now since $c_1\bar{\lambda}\leq \lambda_1$, it follows from
hypothesis \textbf{\textit{M2\;a)}} that both $g_\alpha(\lambda)$
and $r_\alpha(\lambda)$ are nonnegative for all $\lambda \in
[0,c_1\bar{\lambda}]$. On the other hand, from the definitions of
$G_{\bar{\lambda}}$ and $z$ it follows immediately that the
function $h(\lambda)\doteq \seq{F_\lambda Tx,G_{\bar{\lambda}}z}$
for $\lambda\in[0,c_1\bar\lambda]$ is real and non-decreasing and
therefore
\begin{equation}\label{eq:28}
    \int_0^{c_1\bar{\lambda}+}r_\alpha(\lambda)g_\alpha(\lambda)\,d\seq{F_\lambda
   Tx,G_{\bar{\lambda}}z}\geq 0.
\end{equation}
On the other hand, since $h(\lambda)= \seq{Tx,F_\lambda
G_{\bar{\lambda}}z}$ and $F_\lambda
G_{\bar{\lambda}}=G_{\bar{\lambda}}$ for every $\lambda \geq
c_1\bar{\lambda}$, it follows that $h(\lambda)$ is constant for
every $\lambda \geq c_1\bar{\lambda}$ and therefore
\begin{equation}\label{eq:29}
 \int_{c_1\bar{\lambda}+}^{\norm{T}^2+}r_\alpha(\lambda)g_\alpha(\lambda)\,d\seq{F_\lambda
   Tx,G_{\bar{\lambda}}z}= 0.
\end{equation}
From (\ref{eq:28}) and (\ref{eq:29}) we have that
\begin{equation*}
\int_0^{\norm{T}^2+}
   r_\alpha(\lambda)g_\alpha(\lambda)\,d\seq{F_\lambda
   Tx,G_{\bar{\lambda}}z}\geq0,
\end{equation*}
which, by virtue of (\ref{eq:12}), implies that
\begin{equation}\label{eq:30}
\seq{R_\alpha Tx-x,R_\alpha (\bar{y}^{\,\delta}-Tx)}\geq 0.
\end{equation}

By using once again (\ref{eq:Ralfa-galfa}) and (\ref{eq:16})
together with (\ref{eq:30}) it then follows that for every
$\alpha\in (0,\alpha_0)$, for every $\bar{\lambda}\in
\sigma_p(TT^\ast)$ such that $c_1\bar{\lambda}\leq \lambda_1$ and
for every $\delta>0$,
\begin{eqnarray}\label{eq:20}\nonumber
\norm{R_\alpha \bar{y}^{\,\delta}-x}^2 &=& \norm{R_\alpha Tx-x}^2
+\norm{R_\alpha (\bar{y}^{\,\delta}-Tx)}^2 + 2\seq{R_\alpha
Tx-x,R_\alpha (\bar{y}^{\,\delta}-Tx)}\\ \nonumber    &=&
\norm{R_\alpha Tx-x}^2 + \delta^2 \norm{g_\alpha(T^\ast T)T^\ast
G_{\bar{\lambda}}z}^2 + 2\seq{R_\alpha Tx-x,R_\alpha
(\bar{y}^{\,\delta}-Tx)}\\ \nonumber &\geq&\norm{R_\alpha Tx-x}^2
+ \delta^2 \norm{g_\alpha(T^\ast T)T^\ast G_{\bar{\lambda}}z}^2\\
\nonumber  &=& \norm{R_\alpha Tx-x}^2 + \delta^2
\int_0^{{\|T\|^2}\,^+}\lambda\,
g_{\alpha}^2(\lambda)\, d\norm{F_\lambda G_{\bar{\lambda}}z}^2 \\
   &\geq& \norm{R_\alpha Tx-x}^2 + \delta^2
\int_{\bar{\lambda}}^{c_1\,\bar{\lambda}}\lambda
\,g_{\alpha}^2(\lambda)\, d\norm{F_\lambda G_{\bar{\lambda}}z}^2.
\end{eqnarray}

We now consider two different possible cases.

\underline{Case I}: $\alpha \leq \bar{\lambda}$. Since
$c_1\bar{\lambda}\leq\lambda_1$ and $c_1>1$, it follows from
hypothesis \textbf{\textit{M2\;e)}} that
\begin{equation}\label{eq:17}
g_\alpha(\lambda)\geq g_\alpha(c_1\bar{\lambda})\geq
g_\alpha(\lambda_1) \quad\textrm{for every}\quad \lambda \in
[\bar{\lambda},c_1\,\bar{\lambda}].
\end{equation}
On the other hand, from hypothesis \textbf{\textit{M2\;a)}} it
follows that $r_\alpha(\lambda_1)\leq 1$, which implies that
$\lambda_1\,g_\alpha(\lambda_1)\geq 0$ and therefore,
$g_\alpha(\lambda_1)\geq 0.$ It then follows from (\ref{eq:17})
that $g_\alpha^2(\lambda)\geq g_\alpha^2(c_1\,\bar{\lambda})$ for
every $\lambda \in [\bar{\lambda},c_1\,\bar{\lambda}].$ Then,
\begin{eqnarray} \label{eq:21} \nonumber
    \int_{\bar{\lambda}}^{c_1\,\bar{\lambda}}\lambda
\,g_{\alpha}^2(\lambda)\, d\norm{F_\lambda
    G_{\bar{\lambda}}z}^2 &\geq& \bar{\lambda}
\,g_{\alpha}^2(c_1\,\bar{\lambda})\int_{\bar{\lambda}}^{c_1\,\bar{\lambda}}
d\norm{F_\lambda
    G_{\bar{\lambda}}z}^2 \\
    &=& \bar{\lambda}\,g_{\alpha}^2(c_1\,\bar{\lambda}),
\end{eqnarray}
where the last equality follows from the fact that
$\int_{\bar{\lambda}}^{c_1\,\bar{\lambda}} d\norm{F_\lambda
    G_{\bar{\lambda}}z}^2=1$, which is a consequence of the fact that
    $\int_{\bar{\lambda}}^{c_1\,\bar{\lambda}} d\norm{F_\lambda
    G_{\bar{\lambda}}z}^2=\norm{F_{c_1\,\bar{\lambda}}
    G_{\bar{\lambda}}z}^2-\norm{F_{\bar{\lambda}}
    G_{\bar{\lambda}}z}^2$, from the definition of
    $G_{\bar{\lambda}}$, from the fact that $F_\lambda
    F_\mu=F_{\min\{\lambda,\mu\}}$ for every $\lambda,\mu \in \R$
    and the fact that $\norm{G_{\bar{\lambda}}z}=1$.

At the same time, the hypotheses \textbf{\textit{M2\;a)}} and
\textbf{\textit{M2\;c)}} imply that $g_\alpha(\lambda)$ is
monotone decreasing as a function of $\alpha$ for each $\lambda
\in [0,\lambda_1]$. Since $\alpha\leq \bar{\lambda}$ and
$c_1\,\bar{\lambda}\leq \lambda_1$, we then have that
\begin{equation}\label{eq:18} g_\alpha(c_1\,\bar{\lambda})\geq
g_{\bar{\lambda}}(c_1\,\bar{\lambda}),
\end{equation}
and from hypothesis \textbf{\textit{M2\;d)}} we also have that
\begin{equation}\label{eq:18-1}
g_{\bar{\lambda}}(c_1\,\bar{\lambda})\geq
\gamma_2/\bar{\lambda}>0.
\end{equation}
From (\ref{eq:18}) and (\ref{eq:18-1}) we conclude that
\begin{equation}\label{eq:18-2}
g_\alpha^2(c_1\,\bar{\lambda})\geq
\left(\gamma_2/\bar{\lambda}\right)^2.
\end{equation}

Substituting (\ref{eq:18-2}) into (\ref{eq:21}) we obtain
\begin{equation*}
\int_{\bar{\lambda}}^{c_1\,\bar{\lambda}}\lambda
\,g_{\alpha}^2(\lambda)\, d\norm{F_\lambda
    G_{\bar{\lambda}}z}^2 \geq \gamma_2^2/\bar{\lambda},
\end{equation*}
which, together with (\ref{eq:20}) imply that if $\alpha\leq
\bar{\lambda}$, then $\norm{R_\alpha \bar{y}^{\,\delta}-x}^2\geq
(\gamma_2\,\delta)^2/\bar{\lambda}$.

\smallskip
\underline{Case II}: $\alpha>\bar{\lambda}$. In this case, it
follows from hypothesis \textbf{\textit{M2\;c)}} that
$r_\alpha^2(\lambda)\geq r_{\bar{\lambda}}^2(\lambda)$ for every
$\lambda \in (0,\norm{T}^2]$. Then,
\begin{equation*}
\norm{R_\alpha Tx-x}^2=\int_0^{\|T\|^{2\,+}} r_\alpha^2(\lambda)\,
d\norm{E_\lambda x}^2 \geq \int_0^{\|T\|^{2\,+}}
r_{\bar{\lambda}}^2(\lambda)\, d\norm{E_\lambda
x}^2=\norm{R_{\bar{\lambda}} Tx-x}^2,
\end{equation*}
which, together with (\ref{eq:20}) imply that $\norm{R_\alpha
\bar{y}^{\,\delta}-x}^2\geq \norm{R_{\bar{\lambda}} Tx-x}^2$.

\smallskip
Summarizing the results obtained in cases I and II, we can write:
\begin{eqnarray} \label{eq:18-3} \nonumber
\norm{R_\alpha \bar{y}^{\,\delta}-x}^2&\geq& \left\{%
\begin{array}{ll}
    \norm{R_{\bar{\lambda}} Tx-x}^2, & \hbox{if $\alpha >\bar{\lambda}$} \\
    (\gamma_2\,\delta)^2/\bar{\lambda}, & \hbox{if $\alpha \leq
\bar{\lambda}.$} \\
\end{array}%
\right. \\
&\geq&\min\{\norm{R_{\bar{\lambda}} Tx-x}^2,
(\gamma_2\,\delta)^2/\bar{\lambda}\},
\end{eqnarray}
which is valid for every $\alpha \in (0,\alpha_0)$,
$\bar{\lambda}\in \sigma_p(TT^\ast)$ such that
$c_1\bar{\lambda}\leq \lambda_1$ and for every $\delta>0$. Then
\begin{eqnarray*}
\min\left\{\norm{R_{\bar{\lambda}} Tx-x}, \gamma_2
\,\delta/\sqrt{\bar{\lambda}}\right\}&=&
\left(\min\{\norm{R_{\bar{\lambda}} Tx-x}^2,
(\gamma_2\,\delta)^2/\bar{\lambda}\}\right)^{1/2} \\
&\leq& \underset{\alpha \in (0,\alpha_0)}{\inf} \norm{R_\alpha
\bar{y}^{\,\delta}-x}\quad \quad \qquad \qquad \qquad
\parbox{2.5cm}{(by (\ref{eq:18-3}))}  \\
&\leq& \underset{y^\delta \in \overline{B_\delta(Tx)}}{\sup}\; \underset{\alpha \in (0,\alpha_0)}{\inf} \norm{R_\alpha y^\delta-x}\quad \; \parbox{4cm}{(since $\bar{y}^{\,\delta} \in \overline{B_\delta(Tx)}$)} \\
&=& O(\rho(\Theta^{-1}(\delta))) \quad \textrm{for} \quad \delta
\rightarrow
 0^+ \quad \parbox{4cm}{(by hypothesis).}
\end{eqnarray*}

Now, given that $\bar{\lambda}=\alpha(\bar{\delta})$ solves
equation (\ref{eq:14}), from the previous inequality we have that
\begin{equation}\label{eq:22}
\norm{R_{\alpha(\bar{\delta})} Tx-x}=\gamma_2
\,\bar{\delta}/\sqrt{\bar{\lambda}}=O(\rho(\Theta^{-1}(\bar{\delta})))
\quad \textrm{for} \quad \bar{\delta} \rightarrow 0^+,
\end{equation}
which implies that
\begin{equation}\label{eq:23}
\frac{\bar{\delta}}{\rho(\Theta^{-1}(\bar{\delta}))}=O\left(\sqrt{\alpha(\bar{\delta})}\right)
\quad \textrm{for} \quad \bar{\delta} \rightarrow 0^+.
\end{equation}
Since $\delta=\Theta(\Theta^{-1}(\delta))$ it follows from the
definition of $\Theta$ that
$\delta=\sqrt{\Theta^{-1}(\delta)}\,\rho(\Theta^{-1}(\delta))$.
Then, it follows from (\ref{eq:23}) that
$\sqrt{\Theta^{-1}(\bar{\delta})}=O(\sqrt{\alpha(\bar{\delta})})$
for $\bar{\delta} \rightarrow 0^+.$ From this and (\ref{eq:22}) we
then deduce that:
\begin{equation}\label{eq:27}
\norm{R_{\alpha(\bar{\delta})}Tx-x}=O(\rho(\alpha(\bar{\delta})))
\textrm{\;\, for\;\,} \bar{\delta} \rightarrow 0^+ \;\forall\;
\alpha(\bar{\delta})\in \sigma_p(TT^\ast):
c_1\,\alpha(\bar{\delta})\leq \lambda_1.
\end{equation}

\medskip

Now, let $\alpha\in \R^+$ such that $\alpha \leq \underset{j \in
\N}{\max}\{\tilde{\lambda}_j:\tilde{\lambda}_j\leq
\frac{\lambda_1}{c_1}\}$. Then, there exist
$n=n(\alpha)\in\mathbb{N}$ such that $\tilde{\lambda}_{n+1}<\alpha
\leq \tilde{\lambda}_n\le\frac{\lambda_1}{c_1}$. Note here that
$n\to\infty$ if (and only if) $\alpha\to 0^+$.

From hypothesis \textbf{\textit{M2\;c)}} and the fact that
$\tilde{\lambda}_n\in \sigma_p(TT^\ast)$ and
$\tilde{\lambda}_n\leq\frac{\lambda_1}{c_1}$ it follows that
\begin{eqnarray}\label{eq:24}\nonumber
  \norm{R_\alpha Tx-x}^2 &=& \int_0^{\|T\|^{2\,+}} r_\alpha^2(\lambda)\, d\norm{E_\lambda
  x}^2\\ \nonumber
   &\leq& \int_0^{\|T\|^{2\,+}} r_{\tilde{\lambda}_n}^2(\lambda)\,d\norm{E_\lambda
   x}^2\\ \nonumber
   &=&\norm{R_{\tilde{\lambda}_n} Tx-x}^2\\
   &=&O(\rho^2(\tilde{\lambda}_n)),\quad \parbox{4cm}{(by virtue of (\ref{eq:27})).}
\end{eqnarray}
From hypothesis \textbf{\textit{M1}} we have that
$\tilde{\lambda}_n\leq c\,\tilde{\lambda}_{n+1}$ and since $\rho$ is
strictly increasing and positive (by hypothesis
\textbf{\textit{M3}}) it follows that for all $n$ big enough, more
precisely for all $n$ such that $c\,
\tilde{\lambda}_{n+1}\leq\norm{T}^2$,
\begin{equation}\label{eq:25}
\rho^2(\tilde{\lambda}_n)\leq \rho^2(c\,\tilde{\lambda}_{n+1}).
\end{equation}
Now since $c\geq 1$ and $\rho$ is of local upper type $\beta$ for
some $\beta\geq 0$ (hypothesis \textbf{\textit{M3}}), there exists
a positive constant $d$ such that
\begin{equation}\label{eq:26}
\rho(c\,\tilde{\lambda}_{n+1})\leq d\,c^\beta\rho\left(\frac 1 c
\, c\,\tilde{\lambda}_{n+1}\right)=
d\,c^\beta\rho(\tilde{\lambda}_{n+1}).
\end{equation}
 From (\ref{eq:24}), (\ref{eq:25}), (\ref{eq:26}) and from the fact
that $\rho(\tilde{\lambda}_{n+1})< \rho(\alpha)$ it follows that
$\norm{R_\alpha Tx-x}=O(\rho(\alpha))$ for $\alpha \rightarrow
0^+$. Therefore, Lemma \ref{lema:res-reciproco} now implies that
$x \in \mathcal{R}(\rho(T^\ast T))$. This concludes the proof of
the Lemma.\hfill
\end{proof}

\smallskip\smallskip
\begin{rem}
From the definition of qualification (Definition
\ref{def:calif-mathe}) it follows that
\begin{equation*}
\norm{R_\alpha Tx-x}^2 \leq \gamma^2\,\rho^2(\alpha)
\int_0^{+\infty} \rho^{-2}(\lambda)\,d\norm{E_\lambda x}^2.
\end{equation*}
Therefore, in Lemma \ref{lem:sup-inf}, the hypothesis
\textbf{\textit{M1}} and the assumption that $\rho$ be of local
upper type $\beta$ for some $\beta\geq 0$ can be substituted by
the requirement that $\rho(T^\ast T)$ be invertible, or
equivalently, that $\rho^{-2}(\lambda)$ be integrable with respect
to the measure $d\norm{E_\lambda x}^2$ for every $x \in X$.
\end{rem}

\medskip
\medskip
\textit{Proof of Theorem \ref{teo:sat-calmaxima}.} First we note
that from hypotheses \textbf{\textit{M2\;d)}} and
\textbf{\textit{M2\;e)}} it follows easily that
$$
\textbf{\textit{M5}}:\quad
\underset{\lambda\in(0,\|T\|^2]}{\text{sup}}\;
\sqrt\lambda\,|g_\alpha(\lambda)|\ge\frac{b}{\sqrt\alpha}\quad\text{for
every } \alpha\in(0,\alpha_0),
$$
where $b=\gamma_2\sqrt{c_1}$. As in Lemma \ref{lem:sup-inf},
without loss of generality we assume that
$\alpha_0\le\min\{\frac{\lambda_1}{c_1}, \frac{\lambda_1}{k}\}$.
First we will prove that $\psi(x,\delta)\doteq (\rho \circ
\Theta^{-1})(\delta)$ for $x \in X^\rho$ and $\delta \in (0,\Theta
(\alpha_0))$, is an upper bound of convergence for the total error
of $\{R_\alpha\}_{\alpha\in (0,\alpha_0)}$ in $X^\rho$, that is,
we will show that $\psi \in \mathcal{U}_{X^\rho}(\et)$. For every
$r\geq 0$ we define the source sets $X^{\rho,r}\doteq \{x \in
X:x=\rho(T^\ast T)\xi$, $\norm{\xi}\leq r \}$. Let $x \in X^\rho$,
then there exists $r\geq 1$ such that $x \in X^{\rho,r}$. Since
$\Theta$ is continuous and strictly increasing in $(0,\alpha_0)$,
there exists a unique $\tilde{\alpha} \in (0,\alpha_0)$ such that
$x\in X^{\rho,r}$ and $\Theta(\tilde{\alpha})=\frac{\delta}{r}.$
Therefore,
\begin{eqnarray}\label{eq:Trho}\nonumber
  \et(x,\delta) &=& \underset{\alpha \in
(0,\alpha_0)}{\inf}\;\underset{y^\delta \in
  \overline{B_\delta(Tx)}}{\sup}\norm{R_\alpha\,
y^\delta-x} \\ \nonumber
   &\leq & \underset{y^\delta \in
   \overline{B_\delta(Tx)}}{\sup}\norm{R_{\tilde{\alpha}}\,
y^\delta - x} \\
   &\leq& \underset{x \in X^{\rho,r}}{\sup}\;\underset{y^\delta \in
   \overline{B_\delta(Tx)}}{\sup}\norm{R_{\tilde{\alpha}}
\,y^\delta-x}.
\end{eqnarray}
On the other hand, from hypotheses \textit{H1-H4}, the fact that
the function $\rho$ is qualification of $\{R_\alpha\}$, the fact
that $\rho$ trivially \textit{covers} $\rho$ with constant equals
to unity (see \cite{ref:Mathe-Pereverzev-2003}, Definition 2) and
given that $\Theta(\tilde{\alpha})=\frac{\delta}{r}$, it follows
by virtue of Theorem 2 in \cite{ref:Mathe-Pereverzev-2003}, that
there exists a positive constant $K$, independent of $\delta$ such
that
\begin{equation}\label{eq:31}
\underset{x \in X^{\rho,r}}{\sup}\;\underset{y^\delta \in
   \overline{B_\delta(Tx)}}{\sup}\norm{R_{\tilde{\alpha}}
\,y^\delta-x}\leq
K\,\rho\left(\Theta^{-1}\left(\frac{\delta}{r}\right)\right),\;\textrm{for}\;
0<\delta\leq r\,\Theta(\norm{T}^2).
\end{equation}
From (\ref{eq:Trho}) and (\ref{eq:31}) it follows that for every
$\delta \in (0,\Theta(\alpha_0))$,
\begin{equation*}
    \et(x,\delta)\leq
    K\,\rho\left(\Theta^{-1}\left(\frac{\delta}{r}\right)\right)\leq
    K\,\rho(\Theta^{-1}(\delta))= K\, \psi(x,\delta),
\end{equation*}
where the last inequality follows from the fact that $r\geq 1$ and
both $\rho$ and $\Theta^{-1}$ are increasing functions. This
proves that $\psi \in \mathcal{U}_{X^\rho}(\et)$.

Next we will see that $\psi$ satisfies condition \textit{S1} of
saturation on $X^\rho$. From hypotheses \textit{H1-H4},
\textbf{\textit{M4}} and \textbf{\textit{M5}} and the fact that
$\rho$ is maximal qualification of $\{R_\alpha\}$, by virtue of
Theorem 2.3 and Definition 2.2 in \cite{Mathe2004}, it follows that
for every $x^\ast \in X$, $x^\ast \neq 0$ and $x \in X^\rho$ there
exist positive constants $a\doteq a(x,x^\ast)$ and $d=d(x,x^\ast)$
such that
\begin{equation*}
    \frac{\et(x^\ast, \delta)}{\psi(x,\delta)}\geq a \quad
    \forall\, \delta \in (0,d).
\end{equation*}
Then, $\underset{\delta
\rightarrow0^+}{\limsup}\,\frac{\et(x^\ast,\delta)}{\psi(x,\delta)}>0$
for every $x^\ast \in X$, $x^\ast \neq 0$ and $x \in X^\rho$, that
is, $\psi$ satisfies condition \textit{S1} on $X^\rho.$

Also, since $\psi$ does not depend on $x$, it is invariant over
$X^\rho$, i.e., $\psi$ satisfies condition \textit{S2} of
saturation.

It remains to prove that $\psi$ satisfies condition \textit{S3}.
Suppose not. Then, there exist $M\supsetneqq X^\rho$ and
$\tilde{\psi} \in \mathcal{U}_{M}(\et)$ such that
$\tilde{\psi}\mid _{X ^\rho}=\psi$ and $\tilde{\psi}$ satisfies
\textit{S1} and \textit{S2} over $M$. Let $x^\ast \in M\setminus
X^\rho$, $x^\ast \neq 0$. Since $\tilde{\psi} \in
\mathcal{U}_{M}(\et)$ we have that
\begin{equation}\label{eq:des1} \et
\overset{\{x^\ast\}}{\preceq}\tilde{\psi}.
\end{equation}
Since $\tilde{\psi}$ is invariant over $M$ and $X^\rho \subset M$,
it follows that
$\tilde{\psi}\overset{\{x^\ast\},X^\rho}{\preceq}{\tilde{\psi}}$
and since $\tilde{\psi}$ coincides with $\psi$ on $X^\rho$, it
follows that
$\tilde{\psi}\overset{\{x^\ast\},X^\rho}{\preceq}{\psi}$. This
together with (\ref{eq:des1}) imply that
$\et\overset{\{x^\ast\},X^\rho}{\preceq}{\psi}$ and therefore
$\et(x^\ast,\delta)=O(\rho(\Theta^{-1}(\delta)))$ for $\delta
\rightarrow 0^+.$ Lemma \ref{lem:o-O} then implies that there
exists an \textit{a-priori} admissible parameter choice rule
$\tilde{\alpha}:\R^+\rightarrow (0,\alpha_0)$ such that
\begin{equation*}
\underset{y^\delta \in \overline{B_\delta(Tx^\ast)}}{\sup}
\norm{R_{\tilde{\alpha}(\delta)}
y^\delta-x^\ast}=O(\rho(\Theta^{-1}(\delta))) \quad \textrm{for }
\delta \rightarrow 0^+.
\end{equation*}
Then,
\begin{equation*}
\underset{y^\delta \in \overline{B_\delta(Tx^\ast)}}{\sup}\;
\underset{\alpha \in (0,\alpha_0)}{\inf} \norm{R_{\alpha}
y^\delta-x^\ast}=O(\rho(\Theta^{-1}(\delta))) \quad \textrm{for }
\delta \rightarrow 0^+.
\end{equation*}
Finally, Lemma \ref{lem:sup-inf} implies that $x^\ast \in
\mathcal{R}(\rho(T^\ast T))$ and since $x^\ast \neq 0$, we have
that $x^\ast \in X^\rho$, which contradicts the fact that $x^\ast
\in M\setminus X^\rho$. Hence, $\psi$ satisfies condition
\textit{S3} and therefore, $\psi$ is saturation function of
$\{R_\alpha\}$ on $X^\rho$. \hfill\endproof

\smallskip\smallskip
Note that both  Lemma \ref{lem:r-invertible} and Lemma \ref{lem:o-O}
remain true if hypotheses \textbf{\textit{iii)}} and
\textbf{\textit{iv)}} of Theorem \ref{teo:sat-espectral} are
replaced by the requirement that there exists $\rho:(0,\norm{T}^2]
\rightarrow (0,\infty)$ that is qualification of
$\{R_\alpha\}_{\alpha\in (0,\alpha_0)}$ and satisfies the inequality
in the hypothesis \textbf{\textit{M3}} of Theorem
\ref{teo:sat-calmaxima}.
\medskip
\section{Examples} We close our investigation presenting a few
examples of regularization methods possessing global saturation.
For the sake of brevity we shall not provide much details here.

\textbf{Example 1:} The family of Tikhonov-Phillips regularization
operators $\{R_\alpha\}_{\alpha \in (0,\alpha_0)}$ is defined by
(\ref{eq:Ralfa-galfa}) with $g_\alpha(\lambda)=\frac{1}{\lambda +
\alpha}$. It is well known that this family of regularization
operators possesses classical qualification of order $\mu_0=1$. It
can be easily checked that the family $\{g_\alpha\}_{\alpha \in
(0,\alpha_0)}$ satisfies all hypotheses of the Theorem
\ref{teo:sat-espectral} with constants $C\doteq 1$,
$\lambda_1\doteq\norm{T}^2$, $\gamma_1\doteq\frac{1}{2}$,
$c_1\doteq\frac32$, $\gamma_2\doteq\frac25$, $\gamma\doteq\frac 1
2$ and $c\doteq 1.$ Therefore, the function
$\psi(x,\delta)=\delta^{\frac{2}{3}}$ defined for $x \in X_1\doteq
\mathcal{R}(T^\ast T)\setminus\{0\}$ and $\delta>0$ is global
saturation of $\{R_\alpha\}_{\alpha \in (0,\alpha_0)}$ on $X_1$.
\smallskip

\textbf{Example 2:} Given $k\in \mathbb{R}^+$, for
$\alpha,\lambda>0$ let
$$ h_\alpha^k(\lambda)\doteq \begin{cases}
\frac{e^{-\frac{\lambda}{\sqrt\alpha}}}{\lambda}, &\text{for } 0<\lambda<\alpha,\\
\frac{e^{-\sqrt\frac\lambda\alpha}}{\lambda}, &\text{for } \alpha\le\lambda<3\alpha,\\
\frac{e^{-\sqrt\frac\lambda\alpha}}{\lambda}+\frac{\alpha^k}{\lambda^{k+1}},
&\text{for } \lambda\ge 3\alpha,
\end{cases}
$$
and define
$g_\alpha^k(\lambda)\doteq\frac{1}{\lambda}-\alpha^k\sqrt{\lambda}-h_\alpha^k(\lambda)$
for $\lambda>0$, and for $\lambda=0$ define
$g_\alpha^k(0)\doteq\underset{\lambda\to
0^+}{\lim}g_\alpha^k(\lambda)=\frac{1}{\sqrt{\alpha}}.$ It can be
easily verified that for any $\alpha_0>0$,
$\{g_\alpha\}_{\alpha\in (0,\alpha_0)}$ satisfies the hypotheses
\textit{H1}-\textit{H3} and therefore the corresponding collection
of operators $\{R_\alpha\}_{\alpha\in (0,\alpha_0)}$ defined by
(\ref{eq:Ralfa-galfa}) is a family of spectral regularization
operators for $Tx=y$. Hypothesis \textit{H2} is satisfied with
$C\doteq 1+\norm{T}^3\alpha_0^k$. Also, it can easily be proved
that for any $\lambda>0$, $\displaystyle \frac{\lambda^k|1-\lambda
g_\alpha^k(\lambda)|}{\alpha^k}\;=\;O(1)$ for $\alpha\to 0^+$ and
therefore $\{R_\alpha\}_{\alpha \in (0,\alpha_0)}$ has classical
qualification of order $k.$

Now, for $k\geq 1$ and $\alpha>0$, the function
$g_\alpha^k(\lambda)$ is non-increasing. Thus, hypothesis
\textbf{ii.e)} of Theorem \ref{teo:sat-espectral} holds and
$G_\alpha^k\doteq
\norm{g_\alpha^k(\cdot)}_\infty=g_\alpha^k(0)=\frac{1}{\sqrt{\alpha}}$,
which implies immediately that also hypothesis \textit{H4} is
verified. From now on we shall assume $k\ge 1$.

Defining
$$ s_\alpha^k(\lambda)\doteq \begin{cases}
e^{-\frac{\lambda}{\sqrt\alpha}}, &\text{for } 0\leq\lambda<\alpha,\\
e^{-\sqrt\frac\lambda\alpha}, &\text{for } \alpha\le\lambda<3\alpha,\\
e^{-\sqrt\frac\lambda\alpha}+\left(\frac{\alpha}{\lambda}\right)^{k},
&\text{for } \lambda\ge 3\alpha,
\end{cases}$$
it follows that $r_\alpha^k(\lambda)\,=\,1-\lambda
g_\alpha^k(\lambda)\,=\,\alpha^k\lambda^{\frac 3
2}+s_\alpha^k(\lambda)$. Clearly, $r_\alpha^k(\lambda)>0$ for all
$\lambda\geq0$. Now let $\alpha_0\doteq\min\{\frac1
3,\frac{\norm{T}^2}{3}\}$ and $\lambda_1\doteq
\min\{1,\norm{T}^2\}$. It can be shown that $r_\alpha^k(\lambda)\leq
1$ for all $\lambda \in [0,\lambda_1]$ and for all $\alpha\in
(0,\alpha_0)$, i.e., hypothesis \textbf{\textit{ii.a)}} of Theorem
\ref{teo:sat-espectral} is satisfied.

Also, for $0\leq\lambda<\alpha\leq\lambda_1$ we have that
$$r_\alpha^k(\lambda)=\alpha^k\lambda^{\frac 3 2}+e^{-\frac{\lambda}{\sqrt{\alpha}}} >e^{-1},$$
since $\frac{\lambda}{\sqrt{\alpha}}<1$. Thus, hypothesis
\textbf{\textit{ii.b)}} of Theorem \ref{teo:sat-espectral} is
verified with $\gamma_1\doteq e^{-1}.$ Since
$\abs{r_\alpha^k(\lambda)}=r_\alpha^k(\lambda)$ is monotone
increasing with respect to $\alpha$ for all $\lambda\geq 0$,
hypothesis \textbf{\textit{ii.c)}} of Theorem
\ref{teo:sat-espectral} is also satisfied.

On the other hand we have that
$$\alpha
g_\alpha^k(2\alpha)=\frac{1-e^{-\sqrt{2}}}{2}-\sqrt{2}\,\alpha^{\frac
3 2 +k}\geq \frac{1-e^{-\sqrt{2}}}{2}-\sqrt{2}\,3^{-\frac 3 2 -k},$$
since $\alpha\leq\frac 1 3$. Hence hypothesis
\textbf{\textit{ii.d)}} of Theorem \ref{teo:sat-espectral} holds as
well with constants $c_1\doteq 2$ and $\gamma_2\doteq
\frac{1-e^{-\sqrt{2}}}{2}-\sqrt{2}\,3^{-\frac 3 2 -k}>0$ for all
$k\geq 1$.

Finally, for $\lambda\geq 3\alpha$, $$\left(\frac \lambda \alpha
\right)^k \abs{r_\alpha^k(\lambda)}=\left(\frac \lambda \alpha
\right)^k\left(e^{-\sqrt{\frac \lambda
\alpha}}+\alpha^k\lambda^{\frac 3
2}+\left(\frac{\alpha}{\lambda}\right)^k\right)\geq 1,$$ from which
it follows that hypothesis \textbf{\textit{iv)}} of Theorem
\ref{teo:sat-espectral} is satisfied with constants $c\doteq3$ and
$\gamma\doteq 1$. Hence, Theorem \ref{teo:sat-espectral} allows us
to conclude that the function
$\psi_k(x,\delta)=\delta^{\frac{2k}{2k+1}}$ for $x \in X_k\doteq
\mathcal{R}((T^\ast T)^k)\setminus\{0\}$ and $\delta>0$ is global
saturation of $\{R_\alpha\}_{\alpha \in (0,\alpha_0)}$ on $X_k.$

\smallskip
\textbf{Example 3:} Given $\varepsilon>0$, for $\lambda>0$ and
$\alpha\in (0,\alpha_0)$ with $\alpha_0< e^{-1}$, let
$$ h^\varepsilon(\lambda)\doteq \begin{cases}
\alpha, &\text{for } 0\leq\lambda<\alpha,\\
\alpha^{1+\varepsilon}, &\text{for } \lambda\ge \alpha,
\end{cases}
$$
and define $$g_\alpha^\varepsilon(\lambda)\doteq \frac{1+\ln
\alpha}{\lambda \ln
\alpha-\lambda^{-\varepsilon}h^\varepsilon(\lambda)}.$$ It can be
easily checked that $\{g_\alpha^\varepsilon\}_{\alpha\in
(0,\alpha_0)}$ satisfies the hypotheses \textit{H1-H4}. In
particular, hypothesis \textit{H2} is satisfied with $C\doteq 1$.
Therefore $\{R_\alpha\}_{\alpha\in (0,\alpha_0)}$ with $R_\alpha$
as in (\ref{eq:Ralfa-galfa}) is a family of regularization
operators for $Tx=y$. Also it can be shown that for every $\mu>0$,
$$\frac{\lambda^\mu \abs{1-\lambda g_\alpha^\varepsilon(\lambda)}}{\alpha^\mu}\to +\infty \quad \textrm{for}\quad\alpha\to 0^+\quad\textrm{for every}\quad\lambda>0,$$
which implies that $\{R_\alpha\}_{\alpha \in (0,\alpha_0)}$ has
classical qualification of order $\mu_0=0$. Now, the function
$\rho(\alpha)\doteq -(\ln \alpha)^{-1}$ is strictly increasing and
of local upper type $\beta$ for $\beta\doteq 1$ (moreover the
constant $d$ in Definition \ref{def:local-type} can be taken to be
$d\doteq1$) and it can also be proved that $\rho$ is maximal
qualification of $\{R_\alpha\}_{\alpha\in (0,\alpha_0)}$ and
satisfies the inequality in the hypothesis \textbf{\textit{M3}} of
Theorem \ref{teo:sat-calmaxima} with constants $a\doteq 1$ and
$k\doteq 1$.

In this case we have that
$$r_\alpha(\lambda)=\frac{h^\varepsilon(\lambda)+\lambda^{1+\varepsilon}}{h^\varepsilon(\lambda)-\lambda^{1+\varepsilon}\ln \alpha}.$$
Clearly, $r_\alpha^\varepsilon(\lambda)>0$ for all $\lambda\geq
0$. Also, it can be shown that $r_\alpha^\varepsilon(\lambda)\leq
1$ for all $\lambda \in [0,\lambda_1]$ and for all $\alpha\in
(0,\alpha_0)$, where $\lambda_1\doteq \min\{0.6,\norm{T}^2\}$.
Thus, hypothesis \textit{\textbf{M2\;a)}} of Theorem
\ref{teo:sat-calmaxima} is satisfied.

Now, for $0\leq \lambda<\alpha\leq \lambda_1$, we have that
\begin{equation}\label{eq:uno}
r_\alpha^\varepsilon(\lambda)=\frac{\alpha+\lambda^{1+\varepsilon}}{\alpha-\lambda^{1+\varepsilon}\ln
\alpha}\geq\frac{1}{1-\frac{\lambda^{1+\varepsilon}}{\alpha}\ln
\alpha}>\frac{1}{1-\alpha^\varepsilon\ln \alpha},
\end{equation}
since
$\frac{\lambda^{1+\varepsilon}}{\alpha}<\frac{\alpha^{1+\varepsilon}}{\alpha}=\alpha^\varepsilon$.
Since one can easily prove that
\begin{equation}\label{eq:dos}
-\alpha^{\varepsilon}\ln\alpha \leq (3e)^{-1} \quad\textrm{for
all}\quad \alpha>0,
\end{equation}
it follows from (\ref{eq:uno}) and (\ref{eq:dos}) that
$r_\alpha^\varepsilon(\lambda)>(1+\frac 1 {3e})^{-1}$ for $0\leq
\lambda<\alpha\leq \lambda_1$, which implies that hypothesis
\textit{\textbf{M2\;b)}} of Theorem \ref{teo:sat-calmaxima} holds
with $\gamma_1\doteq (1+\frac 1 {3e})^{-1}.$ Since
$\abs{r_\alpha^\varepsilon(\lambda)}=r_\alpha^\varepsilon(\lambda)$
is monotone increasing with respect to $\alpha$ for all
$\lambda\geq 0$, hypothesis \textbf{\textit{M2\;c)}} of Theorem
\ref{teo:sat-calmaxima} is also satisfied.

On the other hand, for $\varepsilon \in (0,1)$ and
$\alpha\in(0,\alpha_0)$, the function
$g_\alpha^\varepsilon(\lambda)$ is non-increasing for $\lambda \in
[\alpha, \lambda_1]$, which implies that hypothesis
$\textbf{\textit{M2\;e)}}$ of Theorem \ref{teo:sat-calmaxima} is
also satisfied.

Assuming $\varepsilon\in(0,1)$, since $s^\varepsilon(\alpha)\doteq
\frac{1+\ln \alpha}{2\ln\alpha-2^{-\varepsilon}}$ is a
non-increasing function
 for $\alpha\in (0,\alpha_0)$ and $2\alpha \leq
\lambda_1\leq 0.6$, we have that
$$\alpha g_\alpha^\varepsilon(2\alpha)=\frac{1+\ln \alpha}{2\ln\alpha-2^{-\varepsilon}}
\geq \frac{1+\ln 0.3}{2\ln 0.3-2^{-\varepsilon}}.$$ Hence
hypothesis \textit{\textbf{M2\;d)}} of Theorem
\ref{teo:sat-calmaxima} is satisfied with constants $c_1\doteq 2$
and $\gamma_2\doteq s^\varepsilon(0.3).$

Finally, for every $\alpha \in (0,\alpha_0)$ the mapping $\lambda
\rightarrow \abs{r_\alpha(\lambda)}^2$, $\lambda \in
(0,\norm{T}^2]$ is convex and therefore hypothesis
\textbf{\textit{M4}} of Theorem \ref{teo:sat-calmaxima} also
holds. Hence, letting $\Theta(t)\doteq
\sqrt{t}\rho(t)=-\frac{\sqrt{t}}{\ln t}$ for $t \in
(0,\norm{T}^2]$, by Theorem \ref{teo:sat-calmaxima} we conclude
that $\psi(x,\delta)\doteq (\rho \circ \Theta^{-1})(\delta)$ for
$x \in X^\rho\doteq \mathcal{R}(\rho(T^\ast T))\setminus
\{0\}\;=\; \mathcal{R}\left(-(\ln(T^\ast T))^{-1}\right)\setminus
\{0\} $ and $\delta \in (0,\Theta
(\alpha_0))\;=\;\left(0,-\frac{\sqrt{\alpha_0}}{\ln{\alpha_0}}\right)$,
is global saturation function of $\{R_\alpha\}_{\alpha \in
(0,\alpha_0)}$ on $X^\rho$.

\smallskip
\textbf{Example 4:} For $\lambda >0$ and $\alpha\in(0,\alpha_0)$,
let $g_\alpha(\lambda)$ be defined as:
$$ g_\alpha(\lambda)\doteq\begin{cases}0, &\text{for
}0\le\lambda<\alpha,\\
\frac{e^{\frac{\lambda}{\ln\alpha}}}{\lambda}, &\text{for
}\lambda\ge\alpha.
\end{cases}
$$
Thus
$$
r_\alpha(\lambda)= \begin{cases}1, &\text{for }0\le\lambda<\alpha,\\
1-{e^{\frac{\lambda}{\ln\alpha}}}, &\text{for }\lambda\ge\alpha.
\end{cases}
$$
It can be immediately shown that $\{g_\alpha\}_{\alpha\in
(0,\alpha_0)}$ satisfies the hypotheses \textit{H1-H4} and
therefore $\{R_\alpha\}_{\alpha\in (0,\alpha_0)}$ with $R_\alpha$
as in (\ref{eq:Ralfa-galfa}) is a family of regularization
operators for $Tx=y$. Also it can be easily checked that
$\{R_\alpha\}_{\alpha\in (0,\alpha_0)}$ has classical
qualification of order $\mu_0=0$. Furthermore, it can be proved
that the function $\rho(\alpha)$ defined by
$$ \rho(\alpha)\doteq \alpha e^{\frac{\alpha}{\ln\alpha}}
$$
is maximal qualification of $\{R_\alpha\}_{\alpha\in
(0,\alpha_0)}$ and all hypotheses of Theorem 4.9 are satisfied.
Hence, letting $\Theta(t)\doteq
\sqrt{t}\rho(t)=t^{\frac32}e^{\frac{t}{\ln t}}$ for $t \in
(0,\norm{T}^2]$, by Theorem \ref{teo:sat-calmaxima} we conclude
that $\psi(x,\delta)\doteq (\rho \circ \Theta^{-1})(\delta)$ for
$x \in X^\rho\doteq \mathcal{R}(\rho(T^\ast T))\setminus \{0\}$
and $\delta \in (0,\Theta
(\alpha_0))\;=\;\left(0,\alpha_0^{\frac32}e^{\frac{\alpha_0}{\ln(\alpha_0)}}
\right)$, is global saturation function of $\{R_\alpha\}_{\alpha
\in (0,\alpha_0)}$ on $X^\rho$.

\medskip
\section{Conclusions}
In this article we have developed a general theory of global
saturation for arbitrary regularization methods for inverse
ill-posed problems. This concept of saturation formalizes the best
global order of convergence that a method can achieve
independently of the smoothness assumptions on the exact solution
and on the selection of the parameter choice rule. Necessary and
sufficient conditions for a methods to have global saturation have
been provided. It was shown that for a method to have saturation
the total error must be optimal in two senses, namely as optimal
order of convergence over a certain set which at the same time,
must be optimal with respect to the error. We have also proved two
converse results and applied the theory to derive sufficient
conditions for the existence of global saturation for spectral
methods with classical qualification of finite positive order and
for methods with maximal qualification. Finally, examples of
regularization methods possessing global saturation were shown.




\bibliographystyle{amsplain}

\end{document}